\documentclass[11pt]{article}
\usepackage[a4paper,margin=1.7cm]{geometry}

\usepackage{url,hyperref,comment}
\usepackage{graphicx,amsfonts,amsmath,amsthm,amssymb}
\usepackage{color}

\newcommand{\C}{\mathcal{C}}

\newcommand{\R}{\mathbb{R}}
\newcommand{\cR}{\mathcal{R}}

\newcommand{\E}{\mathbb{E}}

\newcommand{\1}{\mathbf{1}}
 
\newcommand{\Var}{\operatorname{Var}}

\newcommand{\cS}{\boldsymbol\Delta}
\newcommand{\cF}{\mathcal{F}}

\newcommand{\cQ}{\mathcal{Q}}

\newcommand{\Unif}{\operatorname{Unif}}

\newcommand{\Ber}{\operatorname{Ber}}

\newcommand{\Q}{\mathbb{Q}}

\newcommand{\KL}{\operatorname{KL}}
\newcommand{\TV}{\operatorname{TV}}

\newtheorem{theorem}{Theorem}
\newtheorem{proposition}{Proposition}

\renewcommand{\hat}{\widehat}

\renewcommand{\P}{\mathbb{P}}

\begin{document}
\title{Tighter Confidence Intervals for Rating Systems}
\author{Robert Nowak, Ervin T\'anczos}

\maketitle

\begin{abstract}
  Rating systems are ubiquitous, with applications ranging from
  product recommendation to teaching evaluations. Confidence intervals
  for functionals of rating data such as empirical means or quantiles
  are critical to decision-making in various applications including
  recommendation/ranking algorithms.  Confidence intervals derived
  from standard Hoeffding and Bernstein bounds can be quite loose,
  especially in small sample regimes, since these bounds do not
  exploit the geometric structure of the probability simplex.  We
  propose a new approach to deriving confidence intervals that are
  tailored to the geometry associated with multi-star/value rating
  systems using a combination of techniques from information theory,
  including Kullback-Leibler, Sanov, and Csisz{\'a}r inequalities.
  The new confidence intervals are almost always as good or better
  than all standard methods and are significantly tighter in many
  situations.  The standard bounds can require several times more
  samples than our new bounds to achieve specified confidence interval
  widths. 
\end{abstract}

\section{Introduction}\label{sec:intro}

Multi-star/value rating systems are ubiquitous.  Ratings are used
extensively in applications ranging from recommender systems
\cite{adomavicius2007towards,kwon2008improving} to contests
\cite{tanczos2017kl} to teaching evaluations
\cite{cohen1981student,boysen2015uses}. Key decisions are made based
on comparing functionals of rating histograms such as means and
quantiles. Algorithms for ranking, multi-armed bandits, prefernce
learning, and A/B testing rely crucially on confidence intervals for
these functionals.  This paper develops new constructions for
confidence intervals for multistar rating systems that are often
considerably tighter than most of the known and commonly used
constructions, including Hoeffding,
Bernstein, and Bernoulli-KL bounds.  These are reviewed in Section~\ref{sec:contributions}.

Our main approach begins by considering the construction of confidence
sets in the probability simplex based on finite-sample versions of
Sanov's inequality \cite{cover2012elements} or polytopes formed by
intersecting confidence intervals for the marginal probabilities.
With large probability, these sets include all probability mass
functions that could have generated an observed set of ratings.  An
important aspect of these sets is that they automatically capture to
the intrinsic variability of the ratings.  For instance, if  all of
the ratings are 3 out of 5 stars, then the set is tightly packed in a
corner of the simplex and is effectively much smaller than if the
ratings were uniformly distributed over 1 to 5 stars.  The simplex
confidence sets can then be constrained based on the sort of
functional under consideration (e.g., mean or median).  These
constraints take the form of convex sets in the simplex.  Csisz{\'a}r
inequality \cite{csiszar1984sanov} provides a refinement of Sanov's
bound for such convex sets.

\begin{figure}
	\begin{center}
		\vspace{.2in}
		\includegraphics[width=0.3\textwidth]{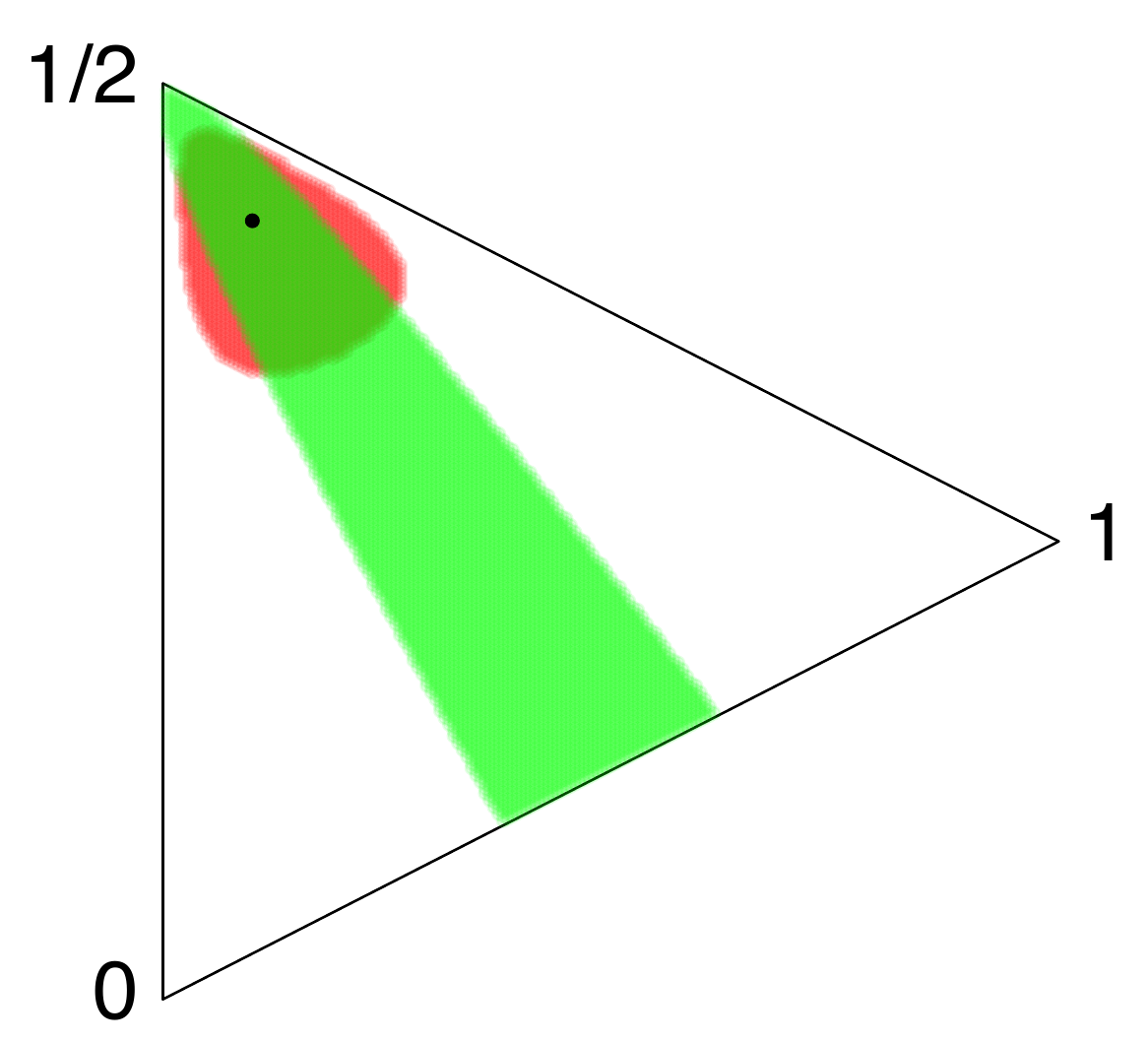}
	\end{center}
	 \caption{Confidence sets based on Sanov (red) and Csisz{\'a}r
           (green) inequalities. Black dot is the empirical
           distribution in this case.The intersection is the set of
           distributions that may have generated the data. \label{fig:simplex} }
\end{figure}

We illustrate how these regions look in the 3-dimensional simplex in
Figure~\ref{fig:simplex}.  Finding the maximum and minimum values for
the functional of interest within the intersection of the Sanov and
Csisz{\'a}r confidence sets yields a new confidence interval for
multistar ratings that is sharper than all common constructions in
almost all cases. Moreover, the new intervals can be easily computed
via optimization, as discussed in Section~\ref{sec:experiments}. A
representative example from a $5$-star rating application (details in
next section) is shown in Figure~\ref{fig:cartoon}.  The empirical
Bernstein (blue) and Bernoulli-KL (red) bounds are the best existing
bounds, but the former performs poorly in low sample regimes and the
latter performs poorly in large sample regimes.  The new bounds
(orange and purple) perform uniformly best over all sample sizes.

\begin{figure}[h]
\begin{center}
\centering
\includegraphics[width = 3in]{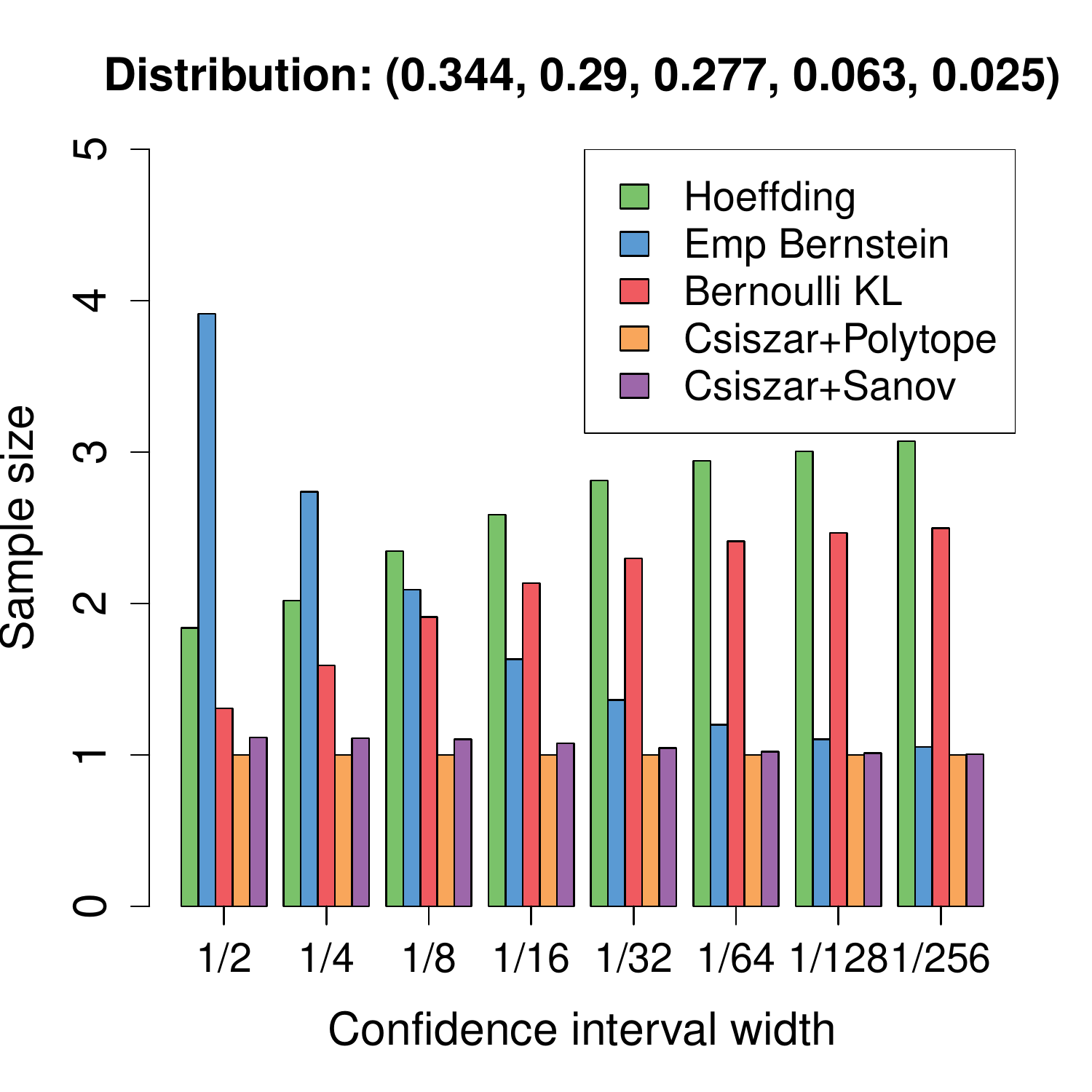}
\end{center}
\vspace{-.2in}
\caption{Comparison of sample sizes for specified confidence interval
  widths using different bounds.  The $1$ to $5$ star distribution
  $(.344, 0.29, 0.277, 0.063, 0.025)$ comes from a real-world contest rating
  dataset.  The sample sizes are normalized relative to the best, so
  the new bounds (best) shown in orange and purple bars are height
  $\approx 1$.  The empirical Bernstein bound (blue) requires about 4
  times more samples that our new bounds in the small sample (large
  interval width) regime.  The Bernoulli-KL bound (red) requires
  almost 3 times more samples in the large sample
  regime. \vspace{-.15in} \label{fig:cartoon}}
\end{figure}

\subsection{Motivating Examples}

Confidence intervals for
ratings are used in ranking applications like the Cartoon Collections
Caption Contest\footnote{www.cartooncollections.com}.  Each week,
contestants submit funny captions for a given cartoon image. Thousands
of captions are submitted, and Cartoon Collections uses crowdsourcing
to obtain hundreds of thousands of ratings for the submissions.  Captions are rated on a
$5$-star scale and ranked according to the average rating each
receives.  The crowdsourcing system uses multi-armed bandit
algorithms
based on confidence intervals to adaptively focus the rating process toward
the funniest captions, yielding a highly accurate ranking of the top
captions.  Better confidence intervals, like the Bernoulli-KL bound,
can significantly improve the accuracy of the ranking, as demonstrated
in \cite{tanczos2017kl}.  The new confidence intervals developed in
this paper offer even greater potential for improvements.  For
example, in a recent contest\footnote{Data courtesy of Cartoon
  Collections.} one caption had the following histogram of
$1$ to $5$ star ratings $(365,308,294,67,27)$. This distribution is quite typical
in this application.  We use this distribution to simulate the rating
process at different sample sizes. Figure~\ref{fig:cartoon} examines
the (normalized) sample
sizes required to achieve confidence intervals of various widths based
on the different bounds. In general, the number of
samples required for an interval of width $W$ scales roughly like $W^{-2}$, and so
we compare the relative number of samples needed by the different
methods. The new bounds developed in this paper, called
Csisz{\'a}r-Polytope and Csisz{\'a}r-Sanov, perform best over all
sample sizes and require $2$-$4$ times fewer ratings than standard bounds
in many cases.

As a second example, consider the two shoes and Amazon ratings shown in
Figure~\ref{fig:shoes}.  The shoe with fewer total ratings has a
slightly higher average rating.  Is the difference in average ratings
statistically significant?  To decide, we need to construct
confidence intervals for the means based on the observed ratings. If
the confidence intervals overlap, then the difference is not
statistically significant. Our desired level of confidence will be
expressed as $1-\delta$, and for the purposes of this example we set
$\delta=0.1$.

\begin{figure}
\begin{center}
\centering
\includegraphics[width=0.5\textwidth]{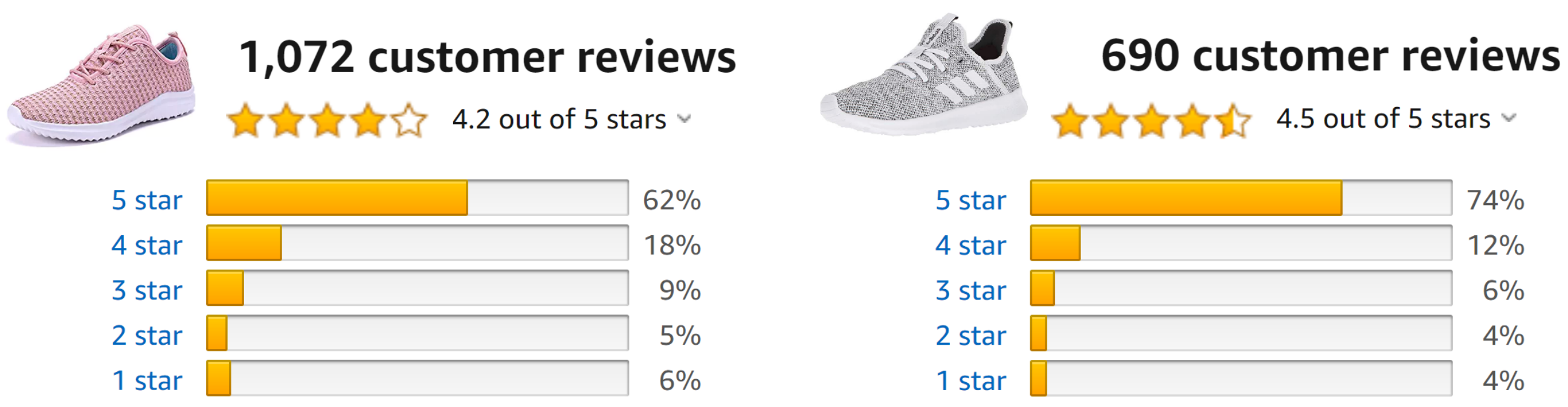}
\end{center}
\vspace{.1in}
\caption{Two shoes with Amazon ratings.\label{fig:shoes}}
\end{figure}

The simplest type of bound is the Hoeffding bound.  This
results in confidence intervals of $[4.10,4.40]$ and $[4.29,4.67]$,
respectively. The Bernoulli-KL bound \cite{garivier2011kl} provides
sharper bounds and yields the intervals $[4.13,4.36]$ and
$[4.34,4.60]$, respectively.  So we can not clearly conclude that Shoe
2 is better than Shoe 1.  In fact, if the observed rating
distributions were the true ones, and assuming equal samples for both
shoes, we would require roughly $1250$ samples per shoe using the
Bernoulli-KL bound. Another
option is to employ empirical Bernstein bounds
\cite{maurer2009empirical}, leading to intervals $[4.12,4.38]$ and
$[4.32,4.63]$.  Again, we can't decide which shoe is better. To do so
would require roughly $1400$ samples per shoe. However, our new bounds
provide the intervals $[4.14,4.35]$ and $[4.36,4.59]$, allowing us to
conclude that with probability at least $0.90$ the true mean rating
for Shoe 2 is larger.  In this case, were the observed rating
distributions true, confidence interval separation would occur at
about $900$ samples per shoe.  So in order to determine that Shoe 2 is
statistically better, the Bernoulli-KL and empirical Bernstein bounds
require about 40\% and 55\% more ratings than our new bounds.  In
extensive experiments in Section~\ref{sec:experiments}, we demonstrate
that all the standard bounds can require many times more samples
than our new bounds to achieve specified confidence interval widths.

\subsection{Related Work}

Since multistar ratings are bounded, standard Hoeffding
bounds can be used to derive confidence intervals.  These bounds do
not account for the bounded and discrete nature of multistar ratings,
nor do they adapt to the intrinsic variance of ratings.  Empirical
versions of Bernstein's inequality
\cite{mnih2008empirical,maurer2009empirical,peel2010empirical,audibert2009exploration,balsubramani2016sequential}
can be used to automatically adapt to the variance of the empirical
process, but as we show these bounds are extremely loose in small
sample regimes. For binary-valued (two-star) ratings, the best known
bounds are based on the Kullback-Leibler (KL) divergence
\cite{garivier2011kl}. Bernoulli-KL bounds automatically adapt to the
variance of binary processes and provide significantly tighter
confidence bounds than standard Hoeffding or Bernstein bounds.  All these
bounds are reviewed in Section~\ref{sec:contributions}.

The Bernoulli-KL bound can be applied to bounded ratings by mapping
the range into $[0,1]$.  These bounds have been shown theoretically and
empirically significantly improve the performance of multi-armed
bandit algorithms \cite{tanczos2017kl}.
However, KL bounds are not well suited to general multistar rating
processes and we show that our new bounds can provide significant
improvements over naive reductions to the Bernoulli KL-type
bounds. Confidence intervals for quantiles are used in many
applications.  For example, \cite{szorenyi2015qualitative} considered
quantile-based multi-armed bandit algorithms and used the
Dvoretzky-Kiefer-Wolfowitz inequality to derive quantile
confidence intervals.  We show that for quantiles other than the
median, our new bounds can yield tighter intervals.

The paper is organized as follows. We set up the problem and review
existing results in Section~\ref{sec:setup}. We define our proposed
confidence bounds tailored for multistar random variables in
Section~\ref{sec:results}, and analyze their accuracy and asymptotic
performance. We also take a moment to review existing methods for
inference about quantiles, given the similarities between those and
the method we propose. Methods for computing the new confidence
intervals and performance comparisons with other confidence intervals
are the focus of
Section~\ref{sec:experiments}. We provide concluding remarks in
Section~\ref{sec:conclusion}.


\section{Problem setup}\label{sec:setup}

Let $\cS_k := \{ p_1,\dots ,p_k :\ p_i>0\ \forall i,\ \sum p_i =1 \}$
denote the probability simplex in $k$-dimensions. Let
$\cF:\ \cS_k \to [0,1]$ be a bounded linear functional mapping from
the probability simplex to $[0,1]$\footnote{We can rescale any bounded
  functional to the interval $[0,1]$.}.  The main focus of
this work is to obtain tight confidence bounds for the value
$\cF (\P ), \P \in \cS_k$, based on an i.i.d. sample
$X_1,\dots ,X_n \sim \P$. We denote the empirical distribution based
on $n$ i.i.d. samples by $\hat{\P}_n$.

If $\cF$ is linear then $\cF (\P ) = \sum_{i\in [k]} p_i w_i$ for any
$\P \in \cS_k$, where $w_i \in \R$ are given weights and $[k]$ denotes
the set $\{1,\dots,k\}$. Furthermore we
can assume w.l.o.g.\ that $\cF=\sum_{i\in [k]} w_i p_i$ with $w_1=0$
and $w_k=1$.  Thus the problem of estimating the value of a linear
functional $\cF (\P )$ is equivalent to estimating the mean of the
random variable $\xi$ defined as $\P (\xi = w_i)=p_i$. In the
discussions that follow, it will be useful to keep both
interpretations of the problem in mind.
Finally, we will also consider cases when $\cF(\P)$ is a quantile, due
to its practical relevance and its similarity to linear functionals.


\subsection{Contributions}\label{sec:contributions}

The most commonly used concentration bounds for the mean of variables
bounded in $[0,1]$ are Hoeffding's inequality, Bernstein's inequality
and the Bernoulli-KL bound (see equation \eqref{eqn:bernoulli} below
and \cite{Concentration_2013} for details). The key difference between
these bounds is the variance information they use. It is
straightforward to see that for a random variable $X\in [0,1]$, the
variance can be upper bounded as follows
\[
\Var (X) = \E (X^2)-\E^2 (X) \leq \E (X) (1-\E (X)) \ .
\]

The Bernoulli-KL bound essentially uses the upper bound above, whereas Hoeffding's inequality further upper bounds the right hand side of the display above by $1/4$. Hence the Bernoulli-KL bound will always be stronger than Hoeffding's bound. However, the variance of $X$ can be smaller than the bound above and Bernstein's inequality explicitly uses this variance information. Therefore it will be tighter then the Binary-KL bound when the variance is indeed smaller.

However, in practice one does not know the variance, and instead has to estimate it from the sample. This gives rise to the empirical Bernstein inequality \cite{maurer2009empirical}, which states that with probability $\geq 1-\delta$
\begin{equation}\label{eqn:emp_bern}
\overline{X}_n - \E (X) \leq \sqrt{\frac{2 \Var_n (X) \log (2/\delta )}{n}} + \frac{7\log (2/\delta )}{3(n-1)} \ ,
\end{equation}
where $\overline{X}_n$ is the empirical mean and $\Var_n (X)$ is the empirical variance.
Asymptotically, the inequality roughly says
\[
\overline{X}_n - \E (X) \leq \sqrt{\frac{2 \Var (X) \log (2/\delta )}{n}} \ ,
\]
or in other words
\begin{equation}\label{eqn:bern}
\P \left( \overline{X}_n - \E (X) >\epsilon \right) \leq \exp \left( -n \frac{2 \epsilon^2}{\Var (X)} \right) \ .
\end{equation}
This is the best exponent we can hope
 for in the limit, since the Central Limit Theorem results in the same
 exponent in the limit as $n\rightarrow \infty$.

Although \eqref{eqn:emp_bern} has good asymptotic performance, its small sample performance is poor. For small $n$ the second term dominates the right hand side of \eqref{eqn:emp_bern}, with the bound often becoming larger than 1, making the inequality vacuous. This property is often undesirable in practice, for instance in the context of bandit algorithms this can lead to wasting a large amount of samples on sub-optimal choices in the early stages of the algorithm.
To overcome this drawback, we propose a confidence bound building on the works \cite{sanov1958probability} and \cite{csiszar1984sanov}. Using their results we construct a confidence region in the probability simplex that contains the true distribution $\P$ with high-probability. Being specialized for distributions on $k$-letter alphabets, these bounds automatically adapt to both the variance and the geometry of the probability simplex. Taking the extreme values of the means of distributions within the confidence region yield the desired confidence bounds for the mean.


\section{Results}\label{sec:results}

In this section we present possible ways of constructing confidence
sets in $\cS_k$ for an unknown distribution $\P$ based on the
empirical distribution $\hat{\P}_n$. For each method we review the
information-theoretic inequality used and describe in detail how it
leads to a confidence set in the simplex $\cS_k$.
The confidence sets presented in Sections~\ref{sec:sanov} and \ref{sec:polygon} are not designed with any specific functional in mind. In Section~\ref{sec:linear} we tailor these regions to specifically work well when $\cF$ is linear. Finally, we also briefly mention the case when $\cF$ is a quantile, and how it relates to the case of linear functionals in Section~\ref{sec:quantiles}.


\subsection{The Sanov-ball}\label{sec:sanov}

Sanov's theorem \cite{cover2012elements} is a natural choice to construct a confidence region for $\P$.
\begin{theorem}[Theorem 11.4.1 of \cite{cover2012elements}]\label{thm:Sanov}
Let $E$ be any subset of the probability simplex  $\cS_k$. Then
\[
\P (\hat{\P}_n \in E)\leq {n+k-1 \choose k-1} \exp \left( -n\inf_{\Q \in E} \KL (\Q ,\P ) \right) \ .
\]
\end{theorem}

We can re-write this result as
\[
\P \left( \KL (\hat{\P}_n ,\P) >z \right) \leq {n+k-1 \choose k-1} e^{-nz} \ ,
\]
which leads to the confidence region
\[
\left\{ \Q :\ \KL (\hat{\P},\Q ) \leq \frac{\log \left( {n+k-1 \choose k-1}/\delta \right)}{n} \right\} \ .
\]
Next, consider an improvement of Sanov's Theorem.
\begin{theorem}\cite{improved_sanov}\label{thm:improved_sanov}
For all $k,n$ 
\begin{align*}
\P & \left( \KL (\hat{\P}_n,\P )>z \right) \\
& \leq \min \Bigg\{ \frac{6e}{\pi^{3/2}} \left( 1+ \sum_{i=1}^{k-2} \Big(\sqrt{\frac{e^3 n}{2\pi i}}\Big)^i \right) e^{-nz}, \\
& \quad 2(k-1)e^{-nz/(k-1)} \Bigg\}
\end{align*}
\end{theorem}
Generally speaking, the first term in the bound is smaller than the second \footnote{In particular, it can be shown that the second term is better whenever $k\leq \sqrt[3]{\tfrac{e^3}{8\pi} n}$, see \cite{improved_sanov}.} when the sample size $n$ is on the same order or lower than the alphabet size $k$. Since in this work we are primarily concerned with situations when the alphabet size is relatively small, we use the second term in the inequality above. This leads to the confidence region
\begin{equation}\label{eqn:sanov_region}
\C_{\rm Sanov} := \left\{ \Q : \KL (\hat{\P},\Q ) \leq \frac{(k-1) \log \left( \frac{2(k-1)}{\delta} \right)}{n} \right\} \ .
\end{equation}
Roughly speaking, this improves a $\log n$ factor to a $\log k$ factor in the cutoff threshold, compared to the one we would get using Theorem~\ref{thm:Sanov}.


\subsection{Confidence Polytope}\label{sec:polygon}

Another simple approach is to construct confidence bounds for the
marginal probabilities $p_j ,\ j\in [k]$ and combine them with a union
bound. For each $j$ let $\hat{p}_j$ denote the empirical frequency of
$j$.  Since  $n\hat{p}_j$ is the sum of independent $\Ber (p_i)$ samples, we
can use the Bernoulli-KL inequality \cite{garivier2011kl}
\begin{equation}\label{eqn:bernoulli}
\P (\KL (\hat{p}_j ,p_j)>z)\leq 2\exp (-nz) \ .
\end{equation}
This leads to the confidence-polytope
\begin{equation}\label{eqn:polygon_region}
\C_{\rm Polytope} := \left\{ \Q : \KL (\hat{p}_j ,q_j ) \leq
  \frac{\log (2k/\delta )}{n} \, , \forall j\in [k] \right\} \ .
\end{equation}

Note that it is not true in general that $\C_{\rm Sanov}$ contains
$\C_{\rm Polytope}$ or vice-versa, and in fact most often neither one
is contained in the other. For one, these sets have different
geometries.
Furthermore, the Bernoulli-KL inequality (and $\C_{\rm Polytope}$ as a consequence) is essentially unimprovable, but there still might be room for improvement in \eqref{thm:improved_sanov} (see the discussion in \cite{improved_sanov}). Therefore, which confidence region performs better depends on the functional $\cF$ and the true distribution $\P$.

That being said, in all numerical experiments presented in Section~\ref{sec:experiments} the bounds derived from $\C_{\rm Polytope}$ consistently beat those derived from $\C_{\rm Sanov}$.


\subsection{Linear functionals}\label{sec:linear}

The main tool we use to construct confidence regions when $\cF$ is linear is Csisz\'ar's theorem \cite{csiszar1984sanov}\footnote{For sake of completeness we include the proof of this theorem in the Supplementary Material.}:
\begin{theorem}
\label{thm:csiszar}
If $E$ is a convex subset of the probability simplex, then
\[
\P (\hat{\P}_n \in E)\leq \exp \left( -n\inf_{\Q \in E} \KL (\Q ,\P ) \right) \ .
\]
\end{theorem}
This theorem can be viewed as a sharpening of Sanov's theorem for convex sets, or as a generalization of the Bernoulli-KL inequality (as we illustrate in Proposition~\ref{prop:geometry} below).
Denote the level sets of the functional $\cF$ by
\[
\cR_m := \left\{ \Q :\ \cF (\Q ) = m \right\} \ ,
\]
and let
\[
B (\Q ,z) = \{ \Q' :\ \KL (\Q' ,\Q )<z \}
\]
denote the KL-ball of radius $z$ around distribution $\Q$. With this we can define the confidence region
\[
\C_{\cF}(\hat{\P}_n,z) = \left\{ \Q :\ \cR_{\cF (\hat{\P}_n)} \cap B(\Q ,z) \neq \emptyset \right\} \ .
\]
We have the following guarantee for this confidence region:
\begin{proposition}\label{prop:linear}
If $\hat{\P}_n$ is the empirical distribution of an i.i.d.\ sample coming from true distribution $\P$, then
\[
\P \left( \P \notin \C_{\cF}(\hat{\P}_n,z) \right) \leq 2 e^{-nz} \ .
\]
\end{proposition}
\begin{proof}
By the definition, if $\P \notin \C_{\cF}(\hat{\P}_n,z)$ implies that
\[
\cR_{\cF (\hat{\P}_n)} \cap B(\P ,z) = \emptyset \ .
\]
This can be restated as
\[
\cF (\hat{\P}_n) \notin \left[ \min_{\Q \in B(\P ,z)} \cF (\Q ),\ \max_{\Q \in B(\P ,z)} \cF (\Q ) \right] \ .
\]
Using the notation $L= \min_{\Q \in B(\P ,z)} \cF (\Q )$ and $U=\max_{\Q \in B(\P ,z)} \cF (\Q )$ we have
\[
\P (\P \notin \C_{\cF} (\hat{\P}_n,z)) = \P \left( \hat{\P}_n \in \left( \bigcup_{z<L} \cR_z \right) \cup \left( \bigcup_{z>U} \cR_z \right) \right) \ .
\]
Note that both regions $\bigcup_{z<L} \cR_z$ and $\bigcup_{z>U} \cR_z$
are convex, since they are unions of `adjacent' hyperplanes.
Using a union bound and Theorem~\ref{thm:csiszar} concludes the proof.
\end{proof}
According to this result
\begin{equation}\label{eqn:linear_region}
\C_{\cF} := \C_{\cF}\left( \hat{\P}_n, \frac{\log (2/\delta )}{n} \right)
\end{equation}
contains $\P$ with probability $1-\delta$.
Note that $\C_{\cF}$ is the KL ``neighborhood'' of the level set
$\cR_{\cF (\hat{\P}_n)}$. As the next result shows, this neighborhood
is widest near the edge of the simplex connecting the corners $(1,0,\dots,0)$
and $(0,\dots,0,1)$ (also see
Firgure~\ref{fig:regions_in_simplex}). This is under the assumption
that the weights of $\cF$ are monotonically increasing (e.g., rating
values of $1$
to $k$ stars).
\begin{proposition}\label{prop:geometry}
Fix a $z>0$ and any $\hat{\P}_n$, and consider the set $\C_{\rm Linear}$. Define
\[
L=\min_{\Q \in \C_{\cF}} \cF (\Q )  \quad \textrm{and}\quad U=\max_{\Q \in \C_{\cF}} \cF (\Q ) \ .
\]
For any $\xi \in [0,1]$ consider the distributions $\P_\xi = (1-\xi,0,\dots ,0,\xi)$ that take value $w_1=0$ with probability $1-\xi$ and value $w_k=1$ with probability $\xi$.
Then the extreme values $L$ and $U$ are uniquely attained by the distributions $\P_L$ and $\P_U$.
\end{proposition}
\begin{proof}
The proof for $L$ and $U$ are similar, so in what follows we focus on $U$.
The claim is a simple consequence of the log-sum inequality. Specifically, consider any two distributions $\P$ and $\Q$. We have
\begin{align*}
\KL & (\cF (\P ),\cF (\Q )) \\
& = \left( \sum_{j\in [k]} w_j p_j \right) \log \frac{\sum_{j\in [k]} w_j p_j}{\sum_{j\in [k]} w_j q_j} \\
& \quad + \underbrace{\left( 1- \sum_{j\in [k]} w_j p_j \right)}_{=\sum_{j\in [k]} (1-w_j) p_j} \log \frac{1-\sum_{j\in [k]} w_j p_j}{1-\sum_{j\in [k]} w_j q_j} \\
& \leq \sum_{j\in [k]} w_j p_j \log \frac{p_j}{q_j} + \sum_{j\in [k]} (1-w_j) p_j \log \frac{p_j}{q_j} \\
& = \KL (\P ,\Q ) \ ,
\end{align*}
by applying the log-sum inequality for both terms on the right side of the first line separately. The inequalities are only tight when $w_j p_j = w_j q_j$ and $(1-w_j) p_j = (1-w_j) q_j$ $\forall j\in [k]$ respectively. This can only happen if $\P \equiv \Q$.
Using this inequality with any $\P \in \cR_{\cF (\hat{\P}_n)}$ and any $\Q \in \cR_U$ implies that the intersection between $\C_{\cF}$ and $\cR_U$ is the single point $\P_{U}$, and the claim is proved.
\end{proof}


\begin{figure}[h]
\begin{center}
\includegraphics[width = 2.5in]{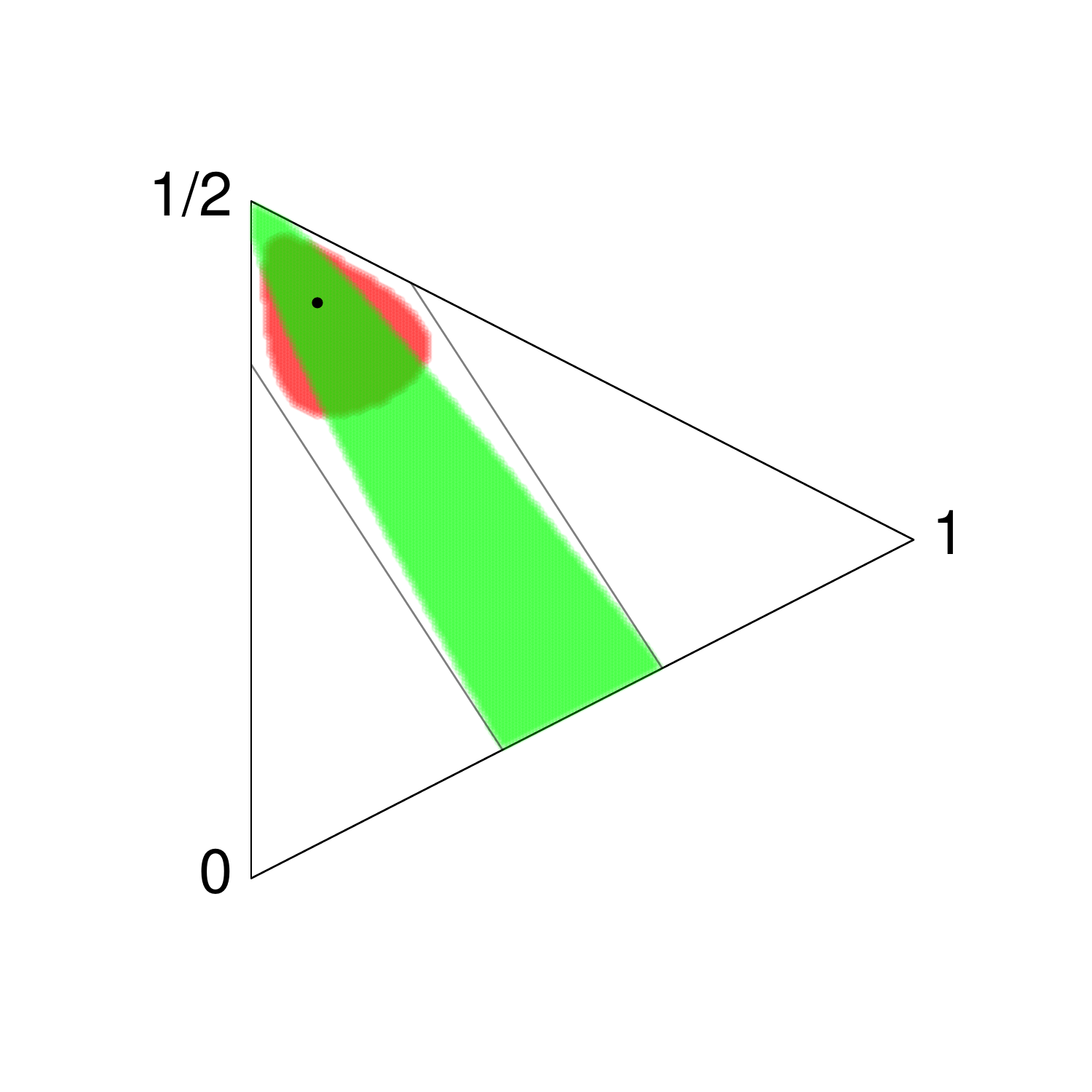}\quad
\includegraphics[width = 2.5in]{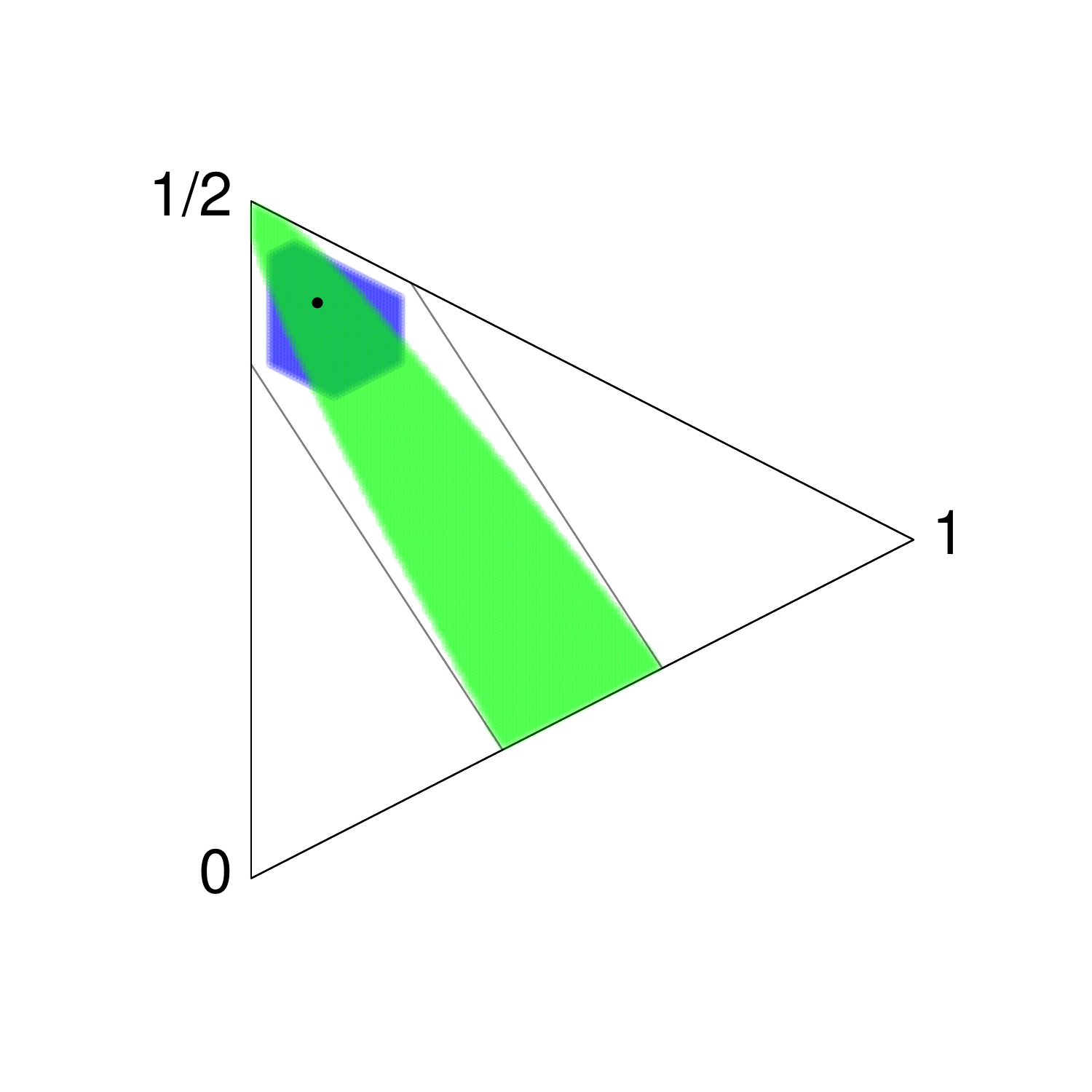} \vspace{-.3in}
\caption{ \ $\C_{\rm Csiszar+Sanov}$ (left) and $\C_{\rm
    Csiszar+Polytope}$ (right) for $\hat{\P}_n=(1/10,8/10,1/10), n=100,
  \delta =0.05$), with $\C_{\cF}$ indicated by the green region. The weights of the linear functional are
  $w_1=0,w_2=1/2,w_3=1$. The plots also include the level sets $\cR_L,
  \cR_U$ (black lines), where $U$ and $L$ are the values defined in
  Proposition~\ref{prop:geometry}. The plots illustrate how
  intersecting $\C_{\rm Sanov}$ or $\C_{\rm Polytope}$ with $\C_{\cF}$
  makes the former regions narrower in the direction perpendicular to
  $\cR_{\cF (\hat{\P}_n)}$ and improve the bounds as a
  result.\vspace{-.15in} \label{fig:regions_in_simplex}}
\end{center}
\end{figure}

Proposition~\ref{prop:geometry} shows in exactly what sense Theorem~\ref{thm:csiszar} is a generalization of \eqref{eqn:bernoulli}: the confidence bounds derived for $\cF (\P )$ using $\C_{\cF}$ are the same as applying \eqref{eqn:bernoulli} to the bounded random variable $\xi$ defined as $\P (\xi = w_i) =p_i$.
However, Proposition~\ref{prop:geometry} also shows that incorporating information about where $\hat{\P}_n$ lies within the simplex might lead to smaller confidence regions, since $\C_{\cF}$ is widest near  the edge of the simplex connecting the corners $(1,0,\dots,0)$
and $(0,\dots,0,1)$, but is potentially narrower elsewhere. A natural way to do this is by intersecting $\C_{\cF}$ with either $\C_{\rm Sanov}$ or $\C_{\rm Polytope}$.
We denote the intersected regions by $\C_{\rm Csiszar+Sanov}$ and
$\C_{\rm Csiszar+Polytope}$ respectively. Naturally, in order to
maintain the same confidence level we need to combine the two regions
using a union bound. We illustrate these regions in a
3-dimensional simplex in Figure~\ref{fig:regions_in_simplex}.

\subsubsection{Asymptotic performance}\label{sec:asymptotics}

The proposition below shows that when we apply Theorem~\ref{thm:csiszar} in the context of linear functionals, the exponent in the bound is equal to what we would get from the central limit theorem. This shows that Theorem~\ref{thm:csiszar} is asymptotically tight.
Based on this, we expect that the confidence bounds derived from both $\C_{\rm Csiszar+Sanov}$ and $\C_{\rm Csiszar+Polytope}$ have optimal asymptotic performance. In Section~\ref{sec:experiments} we illustrate that these bounds enjoy very good performance across all sample sizes.

The full proof of the proposition below can be found in the Supplementary Materials. The high-level
argument is that when $\epsilon$ is small, the minimizer of
$\min_{\Q \in E} \KL (\Q ,\P )$ will be close to $\P$. When $\P$ and
$\Q$ are close, $\KL (\Q ,\P ) \approx \chi^2 (\Q ,\P )$. Minimizing
the chi-squared divergence instead of the KL-divergence on $E$ would
precisely give the value $\epsilon^2 /(2 \Var_\P (\cF ))$. The
proposition shows that the exponent behaves like that of the Bernstein bound
in equation~(\ref{eqn:bern}).
\begin{proposition}\label{prop:asymptotics}
Let $\cF (\P ) = \sum_{j\in [k]} w_j p_j$ be a linear functional. Let $\epsilon >0$ and define $E= \{ \Q :\ \cF (\Q) >\cF (\P ) +\epsilon \}$. For $\epsilon$ small enough, the exponent in Theorem~\ref{thm:csiszar} can be bounded as
\[
\inf_{\Q \in E} \KL (\Q, \P ) \geq \frac{2 \epsilon^2}{\Var_\P (\cF )} - O(\epsilon^3 ) \ ,
\]
where $\Var_\P (\cF ) = \sum_{j\in [k]} w_j^2 p_j - \big( \sum_{j\in [k]} w_j p_j \big)^2$.
\end{proposition}


\subsection{Quantiles}\label{sec:quantiles}

We now take a moment to review the problem of estimating quantiles of a discrete random variable.

The $\tau$-quantile of a random variable $X$ is defined as
\[
\cQ_\tau (X) = \inf \{ x:\ \tau \leq F_X(x) \} \ ,
\]
where $F_X(x)=\P (X\leq x)$ is the CDF. Without loss of generality we assume $X$ takes values in $[k]= \{ 0,\tfrac{1}{k-1},\dots ,1\}$\footnote{For random variables $X$ and $Y$ such that $Y=f(X)$ then $\cQ_\tau (Y) = f(\cQ_\tau (X))$.}. 

The standard method for constructing quantile confidence bounds is first constructing a confidence band for the CDF, and then taking the extreme values of the quantile among distributions in the CDF band. This approach fits the general strategy advocated in this work.

Perhaps the most well-known method to derive confidence bands for the CDF is the DKWM-inequality \cite{massart1990tight}, which states
\[
\P \left( \sup_x \left| \hat{F}_n(x) - F(x) \right| >z \right) \leq 2\exp \left( -2nz^2 \right) \ ,
\]
where $\hat{F}_n$ is the empirical CDF based on $n$ samples. This method is widely used in practice, see for instance \cite{szorenyi2015qualitative}.

However, there exist confidence bands for the CDF that are uniformly
better than those derived from the DKWM inequality, see
\cite{duembgen2014confidence} and references therein. In the context
of discrete random variables taking finitely many values (e.g.,
multistar ratings), the bounds of \cite{duembgen2014confidence} are equivalent to applying the Bernoulli-KL confidence bound for each point of the CDF (i.e. each point in the set $\{0,\tfrac{1}{k-1},\dots ,\tfrac{k-2}{k-1}\}$), and combining them with a union-bound.

If the union-bound is performed naively with the confidence equally
allocated among the $k-1$ points, then the latter confidence band is
inferior to the one obtained from the DKWM inequality for values $x$
where $F_X(x) \approx 1/2$. However, this drawback can be mitigated by
allocating the confidence in a data-driven way, as described and illustrated in Section~\ref{sec:experiments}.


\section{Computational Methods and Experiments}\label{sec:experiments}


\subsection{Linear functionals}

We demonstrate the performance of the method described in Section~\ref{sec:linear} by numerical experiments. We compute the average number of samples needed for the confidence bound for the mean of level $\delta =0.05$ to reach a certain width, for various methods\footnote{Confidence intervals are restricted to lie within
  $[0,1]$.}  and true distributions.
We performed experiments with $k=3$ and $k=5$ and in each case $w_i =(i-1)/(k-1),\ i=1,\dots ,k$. We choose a number of true distributions from the simplex representative of key geometric positions: the midpoint of the probability simplex (the uniform distribution), and midpoints of lower dimensional faces.

Recall that in order to compute the confidence bounds outlined in
Section~\ref{sec:results} we need to solve optimizations $\min_{\Q \in
  \C (\hat{\P}_n)} \cF (\Q )$ and $\max_{\Q \in \C (\hat{\P}_n)} \cF (\Q )$,
where $\cF$ is linear. Since the sets $\C_{\rm Sanov}$ and $\C_{\rm Polytope}$ are convex, solving these optimizations is straightforward.
However, for $\C_{\rm Csiszar+Sanov}$ and $\C_{\rm Csiszar+Polytope}$ the feasible region is itself defined by an optimization, and so the optimizations above become bi-level problems. In particular, for the set $\C_{\rm Csiszar+Sanov}$ we need to solve
\begin{align*}
{\min /\max}_{q_1,\dots ,q_k} & \ \sum_{i\in [k]} w_i q_i \quad \textrm{s.t.} \\
& q_i \geq 0 \forall i\in [k], \sum_{i\in [k]} q_i =1 \ , \\
& \KL (\hat{\P}_n,\Q ) \leq z \ , \\
& \min_{\P' \in \cR_{\cF (\hat{\P}_n)}} \KL (\P' ,\Q ) \leq z' \ ,
\end{align*}
where $z,z' \in \R_+$ are chosen such that both $\C_{\rm Sanov}$ and $\C_{\cF}$ have confidence $\delta /2$ (see \eqref{eqn:sanov_region} and \eqref{eqn:linear_region}). The problem for $\C_{\rm Csiszar+Polytope}$ is analogous.

We solve the problem above using a binary search. Let $u\in [0,1]$ be a fixed value, and suppose we want to decide whether or not $\min_{\Q \in \C_{\rm Csiszar+Sanov}} \cF (\Q ) \leq u$. Deciding this is equivalent to solving
\begin{align*}
\min_{\P' ,\Q} & \KL (\P' ,\Q ) \quad \textrm{s.t.} \\
& q_i \geq 0 \forall i\in [k], \sum_{i\in [k]} q_i =1 \ , \\
& {p'}_i \geq 0 \forall i\in [k], \sum_{i\in [k]} {p'}_i =1 \ , \\
& \KL (\hat{\P}_n,\Q ) \leq z \ , \\
& \sum_{i\in [k]} w_i {p'}_i = \cF (\hat{\P}_n) \ , \\
& \sum_{i\in [k]} w_i q_i = u \ .
\end{align*}
This is a minimization of a convex function subject to convex
constraints, so it can be easily solved with standard solvers. We can
combine this with a binary search to find $\min_{\Q \in \C_{\rm
    Csiszar+Sanov}} \cF (\Q )$. Finding the maximum is analogous. We
implemented all the optimization problems using the R package
CVXR.
\footnote{This implementation may not be the most
  efficient way of computing these confidence bounds. Finding
  the most efficient implementation is an important practical
  consideration.} 
Code for
  the optimization is provided in the Supplementary Materials file code.txt.

The experiments tell a similar story regardless of the true
distribution, therefore we only show a few representative examples in
Figure~\ref{fig:linear} and present more in the Supplementary
Material. In general, the number of
samples required for an interval of width $W$ scales roughly like $W^{-2}$, and so
we compare the relative number of samples needed by the different
methods. We see that our proposed method has the most favorable
sample complexity in almost all cases. Empirical Bernstein
bound starts as a
clear loser, requiring 4-5 times more samples than our new bounds in
the large interval width (small sample) regimes.  However, it matches
our new bounds in the small width (large sample) regimes.  The
Bernoulli-KL bound performs better in large width regimes, but can
become loose in small width regimes.  For example, for the
distribution $(0,0,1/3,1/3,1/3)$ the Bernoulli-KL bound requires
about 4 times more samples that the new bounds to achieve a width of
$1/128$. If the distribution is concentrated on $3$ stars, as is the
case $(0, 0.05, 0.9, 0.05, 0)$, then the poor performance of the Bernoulli-KL bound
is dramatic.
Note that the Bernoulli-KL is best when the true distribution is in
fact Bernoulli, as is the case $(1/2,0,0,0,1/2)$, but our new bounds
are almost as good. The
reason \emph{Csiszar+Sanov} and \emph{Csiszar+Polytope} slightly under
perform in this special case is due to the union bound that arises
when we combine $\C_{\cF}$ with $\C_{\rm Sanov}$ or
$\C_{\rm Polytope}$. This effect could be mitigated by a data-driven
union-bound that allocates most of the confidence budget to $\C_{\cF}$
when $\hat{\P}_n$ is near the edge of the simplex connecting
$(1,0,\dots,0)$ and $(0,\dots,0,1)$. 

\begin{figure}
\begin{center}
\includegraphics[width = 3in]{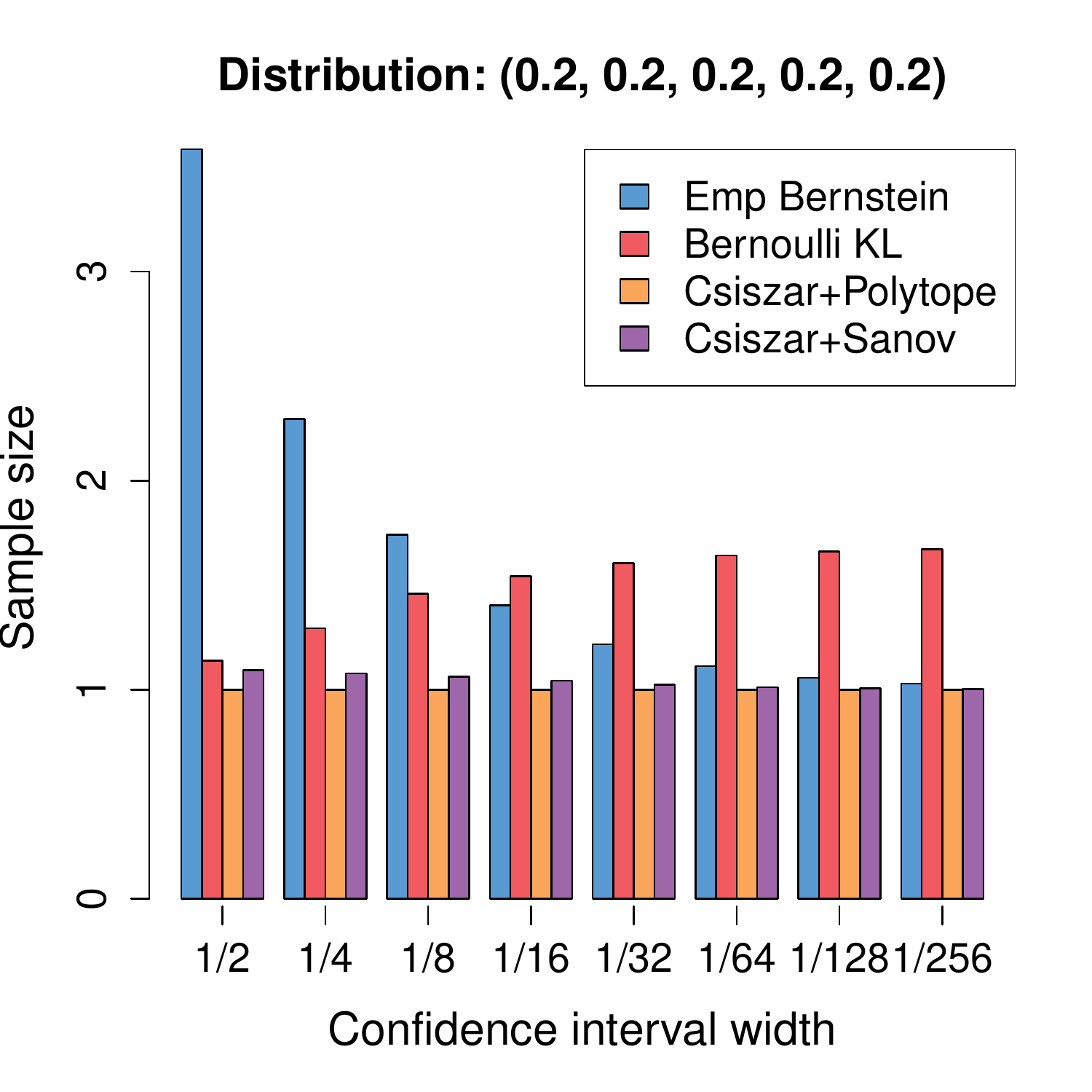} 
\includegraphics[width = 3in]{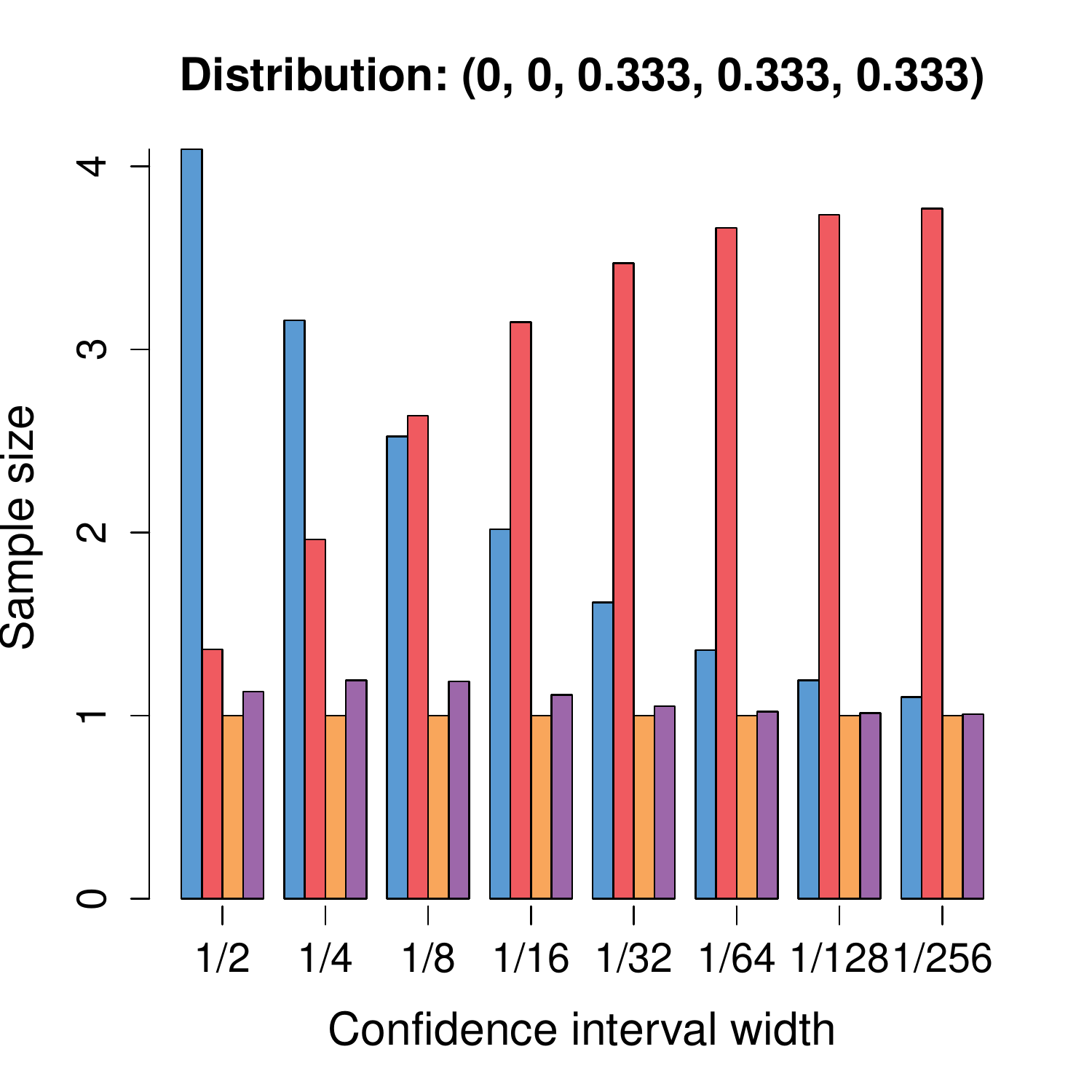}
\includegraphics[width = 3in]{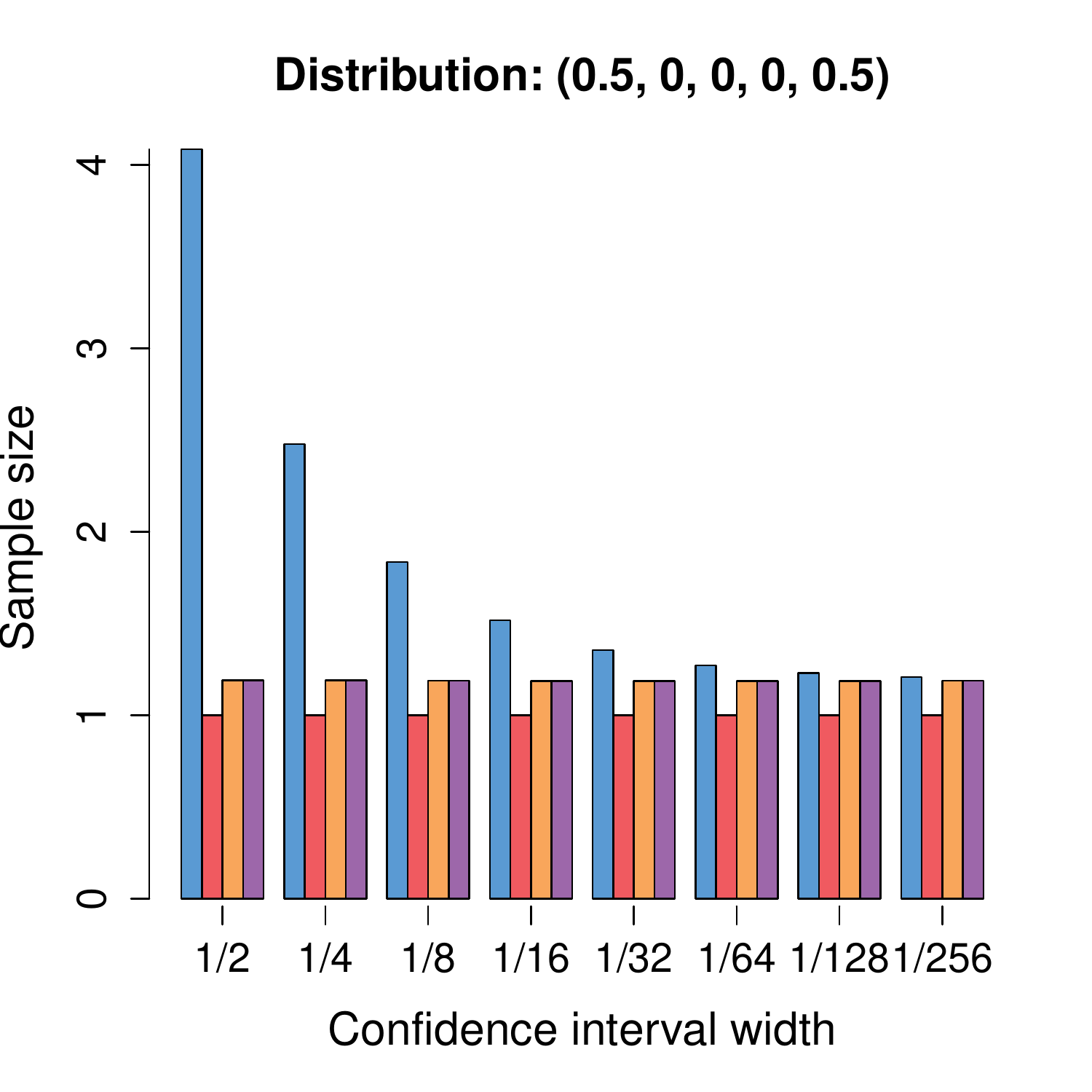} 
\includegraphics[width = 3in]{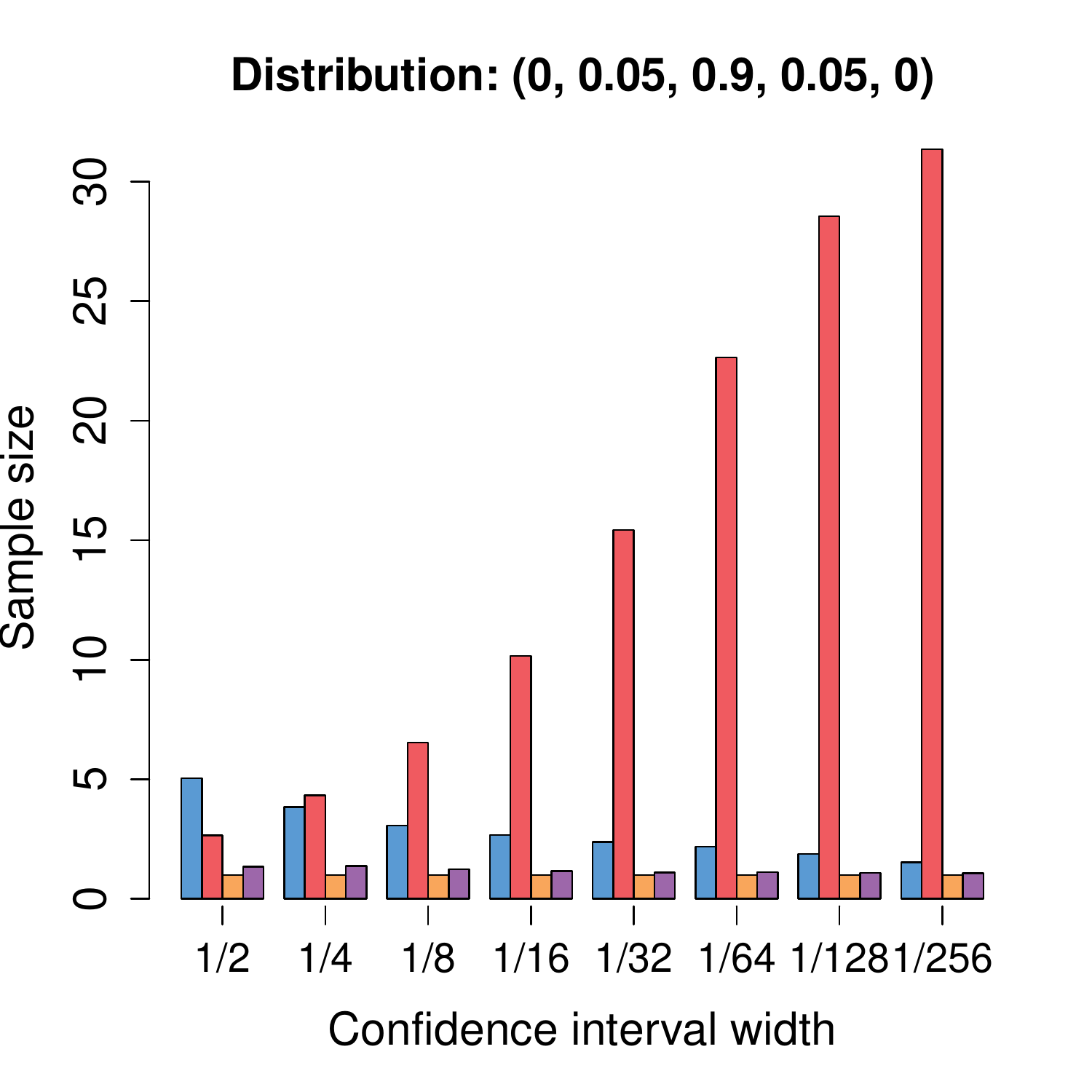}

\vspace{-.1in}
\caption{Average sample size requirements as a function of confidence
  interval width (20 repetitions at each sample size / interval
  width).  The  required sample sizes are very stable (essentially constant) over
  repetitions.  Sample sizes are normalized to the smallest/best (among the various methods) at each width.  The
  empirical Bernstein bound (blue) typically requires several times more
  samples than our new bounds (orange and purple) at small sample
  sizes (large interval widths), but eventually improves as the sample
  sizes increase, as expected.  The Bernoulli-KL bound (red) performs
  comparatively well at small sample sizes, but generallly degrades at
  larger sample size (smaller interval widths), sometimes requiring
  several times more samples than our new bounds. The third distribution
  $(1/2,0,0,0,1/2)$ is an exceptional case, since it corresponds to a
  Bernoulli distribution and the Bernoulli-KL bound is ideal for such cases.\label{fig:linear}}

\end{center}
\end{figure}


\subsection{Quantiles}

\begin{figure}
\begin{center}
\includegraphics[width = 3in]{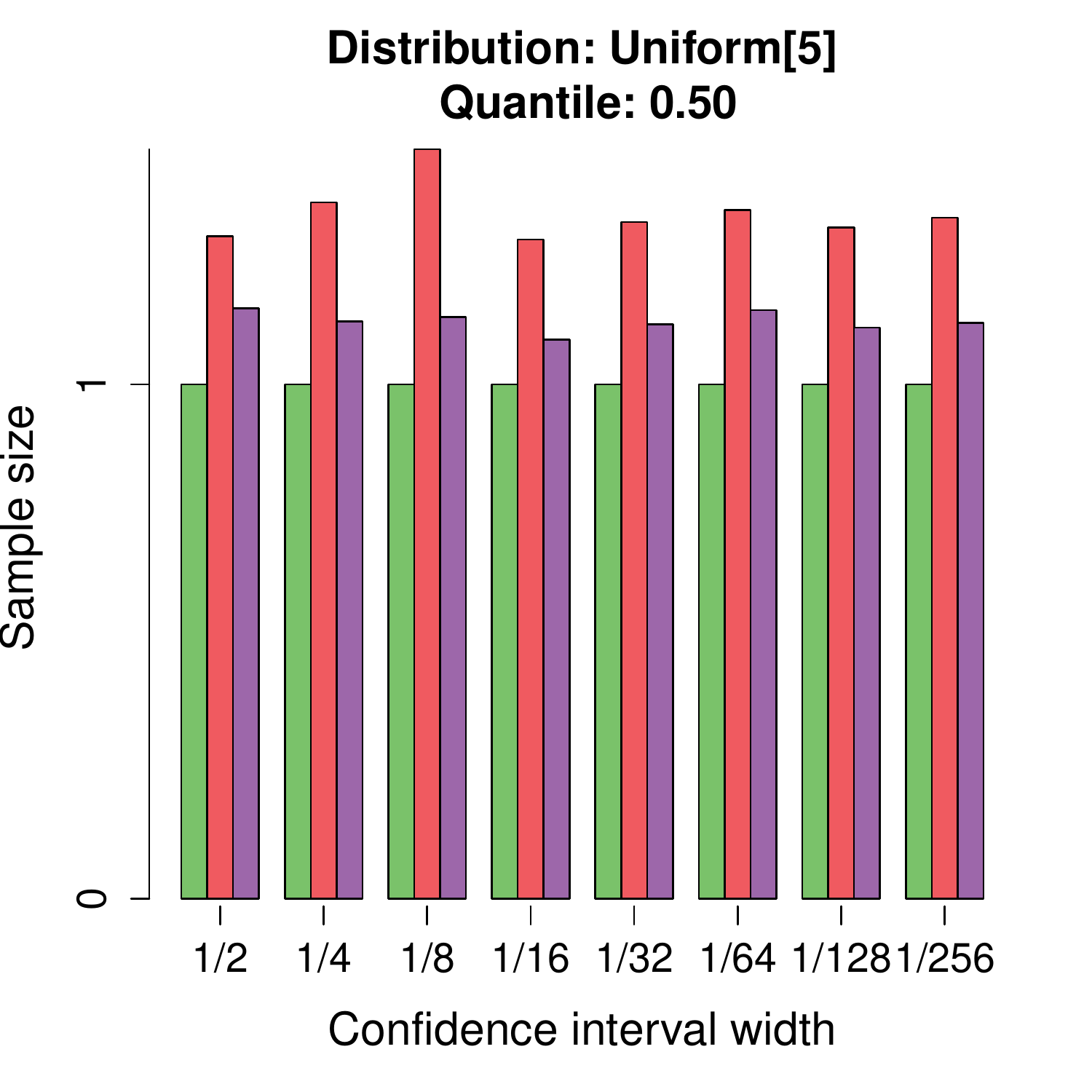}
\includegraphics[width = 3in]{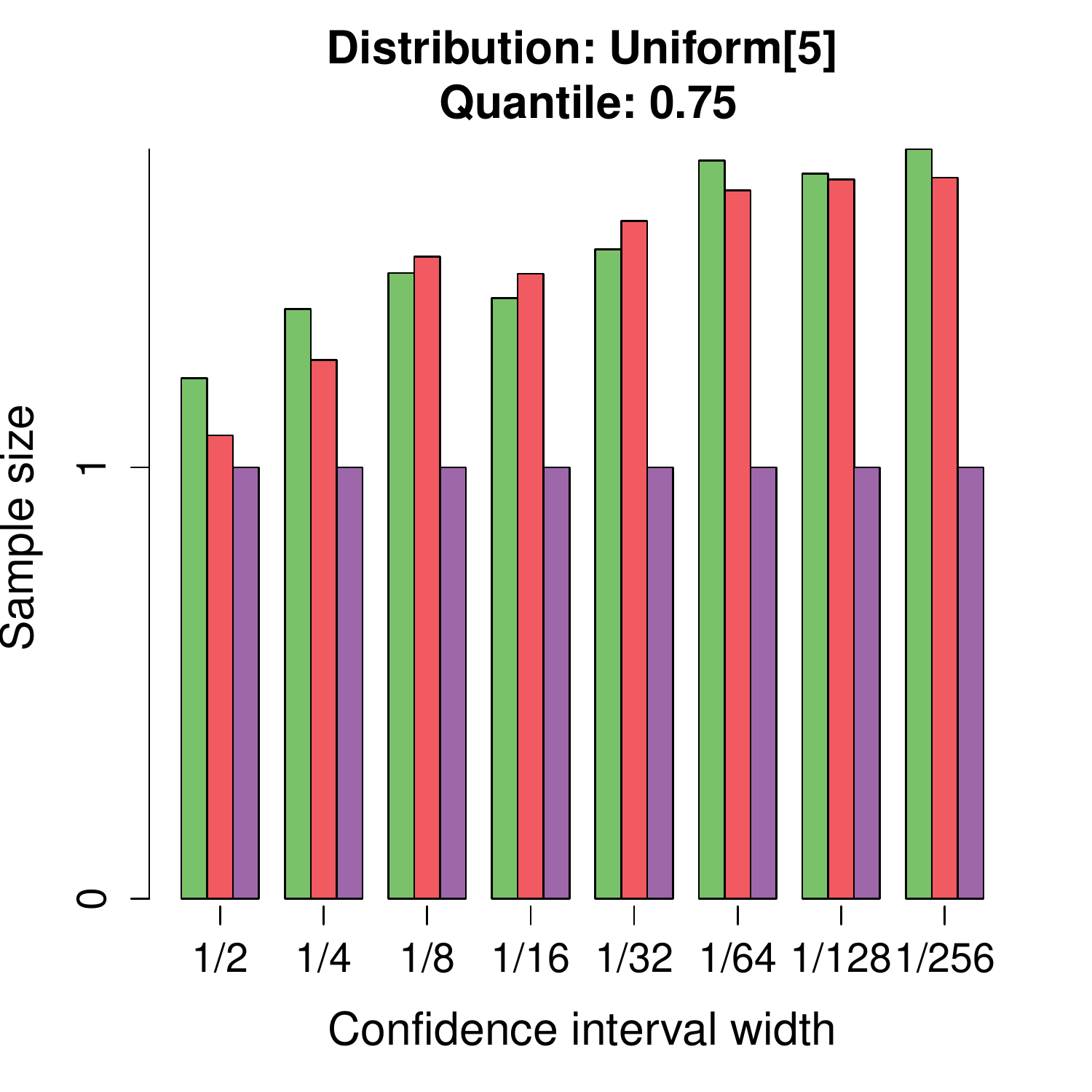}
\includegraphics[width = 3in]{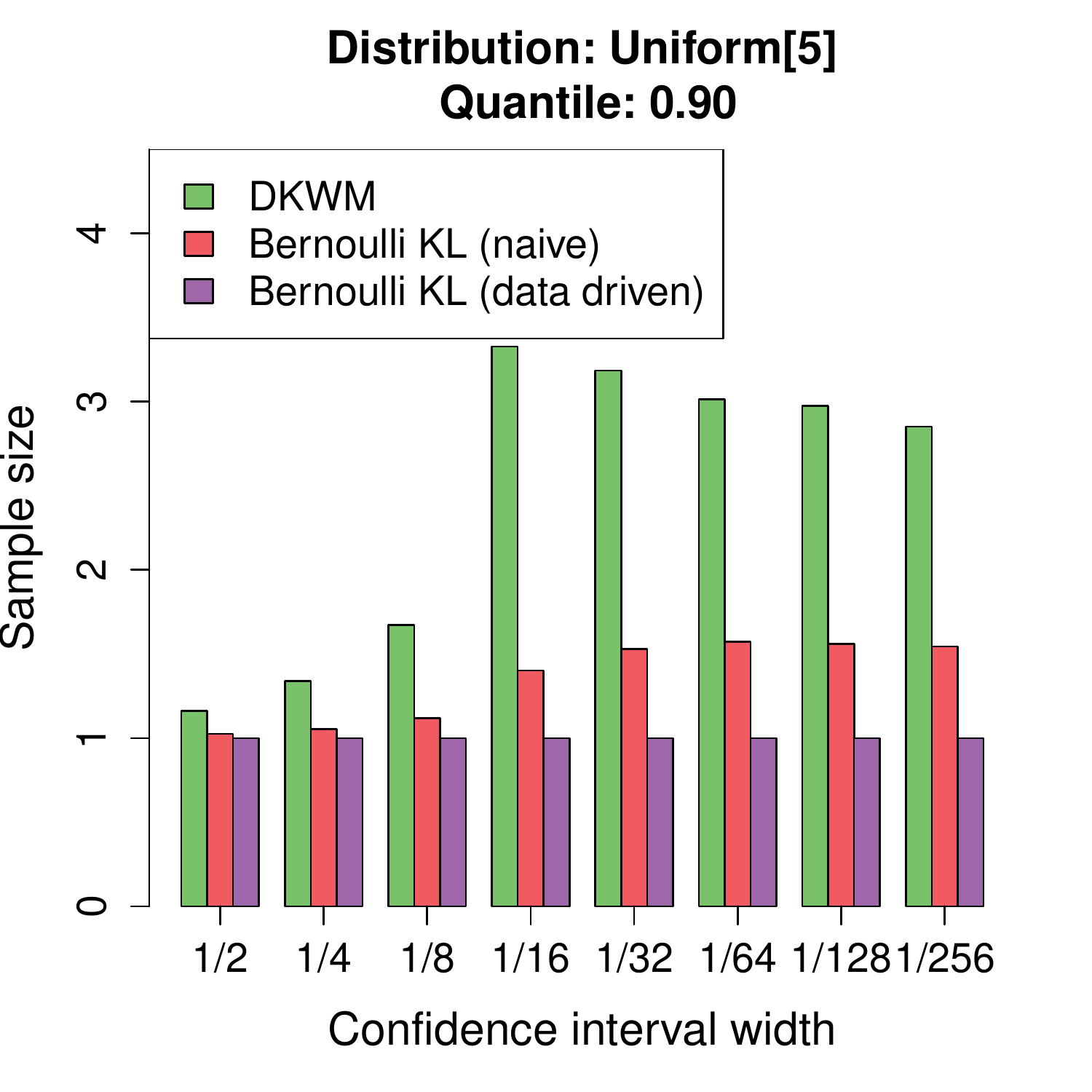}

\caption{Average sample size needed for the width of the confidence
  interval for the quantile to reach a desired level based on 20
  repetitions, for various quantiles.  The required sample sizes are very stable (essentially constant) over
  repetitions. Sample sizes are normalized to the best (among the methods) at each width.  The new Bernoulli-KL CDF bounds perform significantly
  better than the DKWM bound for more extreme quantiles like $0.9$.
\label{fig:quantiles}}
\end{center}
\end{figure}

In this section we compare the performance of CDF bands obtained from the DKWM and Bernoulli-KL inequalities. The width of these bands around the $\tau$-quantile directly influences the derived confidence bounds for the quantile. One possible way of measuring the width is
\[
{\rm Width_\tau} = \sum_{i\in [k]} |\min \{ U_i-\tau , \tau -L_i \}| \1 \{ L_i \leq \tau \leq U_i \} \ ,
\]
where $[L_i,U_i]$ are the confidence bounds for $F_X(i)$, $i\in [k]$.
In Figure~\ref{fig:quantiles} we plot the average sample size needed
for ${\rm Width_\tau}$ to reach a certain value with $\delta =0.05$, for the uniform distribution with $k=5$ and various values for $\tau$. We use two versions of the Bernoulli-KL CDF bounds: one with a naive union bound and one with a data-driven union bound.
The heuristic behind the data-driven union bound is to
assign more
confidence to points of the CDF where the CDF value is close to
$\tau$. We  do this as follows. Define $c_i =
(i-\hat{\tau}+1)^2 \ ,\ i\in [k]$, and $c=\sum_{i\in [k]} 1/c_i$,
where $\hat{\tau}$ is the $\tau$-quantile of $\hat{\P}_n$.  Then we
allocate $\delta /(c \cdot c_i)$ confidence for the bound on $F_X
(i), i\in [k]$. We do not claim that this is the best possible method,
but it yields good results empirically (see  purple bars in Figure~\ref{fig:quantiles}) and we stress that
this approach does yield valid confidence intervals.

Figure~\ref{fig:quantiles} show that the adverse effects of the union bound can be mitigated by using a data-driven method. The Bernoulli-KL method fares almost as well as DKWM when $\tau=0.5$. However, as $\tau$ gets farther away from $0.5$, the benefit of using Bernoulli-KL becomes more and more pronounced. This should come as no surprise, since Bernoulli-KL bound is the tightest possible method for constructing a confidence bound for any fixed point of the CDF.
Since the adverse effects of union bounding can become more pronounced
for larger alphabets, we present numerical experiments for larger
alphabets in the Supplementary Material. 


\section{Conclusion}\label{sec:conclusion}

In this work we illustrated the merit of using information-theoretic
inequalities for constructing confidence bounds for functionals of
multistar random variables. These bounds account for the geometry of
the probability simplex, and as a result exhibit excellent performance
across all sample sizes when compared to other popular bounds in the
literature. Conventional bounds may need up several times more samples
to reach the same confidence interval width as the bounds proposed in
this work.  Although outside the scope of this work, the general
recipe presented here might prove fruitful for functionals other than
linear, such as the variance or higher moments. Extending these
methods to other functionals is a fruitful avenue for
future research.


\bibliographystyle{acm}
\bibliography{Multistar_references}

\appendix


\section{Running shoe example}

The shoes and the ratings in the example of Section~\ref{sec:intro} are \href{https://www.amazon.com/YILAN-Womens-Fashion-Sneakers-Pink-4/dp/B06XZ4VKW9/ref=sr_1_4?ie=UTF8&qid=1547297579&sr=8-4&keywords=sneakers}{Shoe 1} and \href{https://www.amazon.com/adidas-Performance-Womens-Cloudfoam-Running/dp/B0711R2TNB/ref=sr_1_5?ie=UTF8&qid=1547297579&sr=8-5&keywords=sneakers}{Shoe 2}.


\section{Proof of Theorem~\ref{thm:csiszar}}\label{app:csiszar_proof}

We begin with setting up notation. Without loss of generality, assume that $X$ takes values from $[k]$. For any $n$-length sequence $x \in [k]^n$, let $T_x$ denote the \emph{type} of $x$ (the empirical distribution generated by the sequence). Denote the set of all types based on $n$-length sequences by $\mathcal{T}_n$, formally $\mathcal{T}_n = \{ T_x :\ x\in [k]^n \}$. We use the shorthand notation $\P (E) = \P (x:\ T_x \in E)$.

Define the distribution $\overline{\P}(i) = \sum_{x\in [k]^n} \Q (x) T_x (i)$, where $\Q$ is an arbitrary distribution over \emph{$n$-length sequences}. Then
\begin{align*}
\P (E) & = \exp \left( \log \P (E) \right) \\
& = \exp \left( \sum_{x\in [k]^n} \Q (x) \log \P (E) \right) \\
& = \exp \Bigg( \sum_{x\in [k]^n} \Q (x) \log \frac{\P (x)}{\overline{\P} (x)} \\
& \quad + \sum_{x\in [k]^n} \Q (x) \log \frac{\overline{\P} (x) \P (E)}{\P (x)} \Bigg) \ .
\end{align*}
But
\begin{align*}
\sum_{x\in [k]^n} & \Q (x) \log \frac{\P (x)}{\overline{\P} (x)} \\
& = \sum_{x\in [k]^n} \Q (x) \log \prod_{j\in [k]} \left( \frac{\P (j)}{\overline{\P} (j)} \right)^{n T_x (j)} \\
& = \sum_{x\in [k]^n} \Q (x) t \sum_{j\in [k]} T_x (j) \log \frac{\P (j)}{\overline{\P} (j)} \\
& = -n \KL (\overline{\P} ,\P ) \ ,
\end{align*}
so
\begin{align}\label{eqn:second}
\P (E) & = \exp \Bigg( -n \KL (\overline{\P} ,\P ) \nonumber \\
& \quad + \sum_{x\in [k]^n} \Q (x) \log \frac{\overline{\P} (x) \P (F)}{\P (x)} \Bigg) \ .
\end{align}

Now we use a specific choice for $\Q$. Let $\Q (x) = \1 \{ x\in E \} \P (x) /\P (E) := P_E (x)$ for short. Then
\[
\sum_{x\in [k]^n} \Q (x) \log \frac{\overline{\P} (x) \P (E)}{\P (x)} = -\KL (\P_E ,\overline{\P}) \ ,
\]
and so
\[
\P (E) \leq \exp \left( -n \KL (\overline{\P} ,\P ) \right) \ .
\]
If $E$ is convex and $\Q$ is supported on $E$ (note that with the above choice this is true), then $\overline{\P} \in E$ and hence
\[
\P (E) \leq \exp \left( -n \inf_{\P' \in E} \KL (\P' ,\P ) \right) \ .
\]


\section{Proof of Proposition~\ref{prop:asymptotics}}

We begin by providing the road map for the proof. The high-level argument is that when $\epsilon$ is small, the minimizer of $\min_{\Q \in E} \KL (\Q ,\P )$ will be close to $\P$. When $\P$ and $\Q$ are close, $\KL (\Q ,\P ) \approx \chi^2 (\Q ,\P )$. Minimizing the chi-squared divergence instead of the KL-divergence on $E$ would precisely give the value $\epsilon^2 /(2 \Var_\P (\cF ))$.

Carrying out the proof formally requires care, in particular to be able to switch between $\min_{\Q \in E} \KL (\Q ,\P )$ and $\min_{\Q \in E} \chi^2 (\Q ,\P )$.

We begin by upper bounding the order of magnitude of $\min_{\Q \in E} \KL (\Q ,\P )$ in terms of $\epsilon$. This will be necessary to control the error we induce by switching between the two optimizations.

Let $\Q' \in E$ be such that $q'_i/p_i = \lambda w_i + \nu$ with some $\lambda ,\nu >0$. Note that in order for $\Q'$ to be a proper distribution we must have
\begin{equation}\label{eqn:distr}
1 = \sum_{i\in [k]} p_i (\lambda w_i +\nu ) = \lambda \cF (\P) + \nu \ .
\end{equation}
In order for $\Q'$ to be in $E$ we need
\begin{align}\label{eqn:inE}
\epsilon & = \sum_{i\in [k]} (q'_i -p_i) w_i \nonumber \\
& = \sum_{i\in [k]} p_i w_i (\lambda w_i +\nu -1) \nonumber \\
& = \lambda \sum_{i\in [k]} p_i w_i (w_i -\cF (\P )) \nonumber \\
& = \lambda \Var_\P (\cF ) \ ,
\end{align}
where in the third line we used \eqref{eqn:distr}.

From \eqref{eqn:distr} and \eqref{eqn:inE} we can conclude that
\begin{align*}
q'_i /p_i & = \lambda w_i +1 -\cF (\P ) \\
& = 1 + \frac{\epsilon}{\Var_\P (\cF )} (\cF (\P )-w_i) \\
& = 1+O(\epsilon ) \ ,
\end{align*}
for all $i\in [k]$.

Note that clearly $\min_{\Q \in E} \KL (\Q ,\P ) \leq \KL (\Q' ,\P )$. We will upper bound the right hand side when $\epsilon$ is small. In particular, we will use the Taylor expansion of $x\log x$ around $x=1$ with a Lagrange remainder term, i.e.
\[
x\log x = (x-1) + \frac{1}{2} (x-1)^2 - \frac{1}{6 \xi^2} (x-1)^3 \ ,
\]
where $\xi \in (1-a,1+a)$ and $a$ is the radius of the expansion. Since we concluded that $q'_i/p_i$ is close to 1, we can choose $a$ to be some arbitrary constant when $\epsilon$ is small enough.

Using the Taylor expansion above, and the fact that $q'_i/p_i \in (1-O(\epsilon),1+O(\epsilon )) $ we get that
\begin{align*}
\KL (\Q' ,\P ) & = \sum_{i\in [k]} p_i \frac{q'_i}{p_i} \log \frac{q'_i}{p_i} \\
& = \sum_{i\in [k]} p_i \Bigg( \left( \frac{q'_i}{p_i} -1 \right) + \frac{1}{2} \left( \frac{q'_i}{p_i} -1 \right)^2 \\
& \qquad - \frac{1}{6 \xi_i^2} \left( \frac{q'_i}{p_i} -1 \right)^3 \Bigg) \\
& \leq \frac{1}{2} \sum_{i\in [k]} \left( \frac{q'_i}{p_i} -1 \right)^2 + O(\epsilon^3 ) \\
& \leq \frac{1}{2} \left( \sum_{i\in [k]} \frac{{q'}_i^2}{p_i} -1 \right) + O(\epsilon^3 ) \ .
\end{align*}

Plugging in $q'_i/p_i = \lambda w_i +\nu$ and using \eqref{eqn:distr} and \eqref{eqn:inE} we can continue as
\begin{align*}
\KL (\Q' ,\P ) & \leq \frac{1}{2} \left( \sum_{i\in [k]} \frac{{q'}_i^2}{p_i} -1 \right) + O(\epsilon^3 ) \\
& = \frac{1}{2} \left( \sum_{i\in [k]} q'_i \lambda w_i +\nu -1 \right) + O(\epsilon^3 ) \\
& = \frac{1}{2} \lambda \sum_{i\in [k]} q'_i (w_i -\cF (\P ) ) + O(\epsilon^3 ) \\
& = \frac{\epsilon^2}{2 \Var_\P (\cF )} + O(\epsilon^3 ) \ .
\end{align*}

So far we have shown that $\min_{\Q \in E} \KL (\Q ,\P ) \leq \epsilon^2/(2 \Var_\P (\cF )) + O(\epsilon^3 )$. We now use this to switch from the optimization of the KL-divergence to that of the $\chi^2$-distance.

First we use the upper bound above to conclude that the unique minimizer\footnote{We know that $\Q^*$ is unique since $E$ is convex.} to $\min_{\Q \in E} \KL (\Q ,\P )$ denoted by $\Q^*$ is also close to $\P$ in Total-Variation distance. This fact is a simple consequence of Pinsker's inequality:
\[
\TV (\Q^* ,\P ) \leq \sqrt{\KL (\Q^* ,\P )/2} = O(\epsilon ) \ ,
\]
where $\TV (\cdot ,\cdot )$ denotes the Total Variation distance, and on the right side we used $\KL (\Q^* ,\P ) \leq O(\epsilon^2 )$. Denoting the Total Variation ball of radius $z$ around $\P$ by $B_{\TV} (\P ,z)$ we have now shown that $\Q^* \in B_{\TV} (\P ,O(\epsilon ))$.

We are finally in position to formally show the lower bound for $\min_{\Q \in E} \KL (\Q ,\P )$. In particular
\begin{align*}
\min_{\Q \in E} \KL (\Q ,\P ) & = \min_{\Q \in E \cap B_{\TV} (\P ,O(\epsilon ))} \KL (\Q ,\P ) \\
& = \min_{\Q \in E \cap B_{\TV} (\P ,O(\epsilon ))} \Bigg( \frac{1}{2} \left( \sum_{i\in [k]} \frac{q_i^2}{p_i} -1 \right) \\
& \quad - \frac{1}{6 \xi (q_i ,p_i)^2} \left( \frac{q_i}{p_i} -1 \right)^3 \Bigg) \ ,
\end{align*}
using the same Taylor-expansion as before. Note that the Taylor expansion is valid here because we are only considering distributions $\Q$ that are close to $\P$, i.e. $\Q \in E \cap B_{\TV} (\P ,O(\epsilon ))$.

However, for distributions $\Q$ in $B_{\TV} (\P ,O(\epsilon ))$ we have $q_i/p_i -1 = O(\epsilon )$. Hence we can continue as
\begin{align*}
\min_{\Q \in E} & \KL (\Q ,\P ) \\
& = \min_{\Q \in E \cap B_{\TV} (\P ,O(\epsilon ))} \Bigg( \frac{1}{2} \left( \sum_{i\in [k]} \frac{q_i^2}{p_i} -1 \right) \\
& \quad - \frac{1}{6 \xi (q_i ,p_i)^2} \left( \frac{q_i}{p_i} -1 \right)^3 \Bigg) \\
& \geq \min_{\Q \in E \cap B_{\TV} (\P ,O(\epsilon ))} \frac{1}{2} \left( \sum_{i\in [k]} \frac{q_i^2}{p_i} -1 \right) - O(\epsilon^3 ) \\
& \geq \min_{\Q \in E} \frac{1}{2} \left( \sum_{i\in [k]} \frac{q_i^2}{p_i} -1 \right) - O(\epsilon^3 ) \ .
\end{align*}

All that is left to do is to solve the optimization of the $\chi^2$-divergence. In detail, the optimization we need to solve is
\begin{align*}
\min & \frac{1}{2} \left( \sum_{j\in [k]} \frac{q_j^2}{p_j} -1 \right) \ \textrm{s.t.} \\
& \sum_{j\in [k]} q_j =1,\ q_j \geq 0,\ \forall j\in [k]\ ,\\
& \sum_{j\in [k]} w_j (q_j-p_j) = \epsilon\ .
\end{align*}
Taking the derivative of Lagrangian w.r.t. $q_j$ yields
\[
\frac{\partial}{\partial q_j} \mathcal{L}(\underline{q},\lambda ,\nu ,\underline{\eta}) = \frac{q_j}{p_j} -\lambda w_j -\nu -\eta_j \ .
\]
Equating this to zero and rearranging gives an expression for the optimizer $\Q$.

Without loss of generality, we can assume that $\P$ is in the interior of the simplex, since otherwise we would just restate the entire argument in lower dimension. If $\epsilon$ is small enough then the optimizer will satisfy $q_j>0\ \forall j\in [k]$\footnote{We omit a detailed argument here, but this is clear: the optimization problem considered here is searching for an ellipse centered at $\P$ that touches the half-space $E$.}. Thus the KKT optimality conditions give $\eta_j =0$ for all $j\in [k]$. Hence we have that the solution of the optimization $\Q^*$ satisfies
\[
\frac{q^*_j}{p_j} = \lambda w_j + \nu \ .
\]

From this point on we continue the same way as we did at the beginning of the proof to finally conclude
\[
\min_{\Q \in E} \KL (\Q ,\P ) \geq \frac{\epsilon^2}{2 \Var_\P (\cF )} - O(\epsilon^3 ) \ .
\]


\section{Figures for numerical experiments}

\subsection{Linear Functionals}

We present the plots corresponding to the numerical experiments that we omitted from the main body of the paper. The plots shown here correspond to experiments with various values of the true distribution. Regardless, all experiments tell a similar story to the one outlined in the paper.

\begin{figure*}[h]
\begin{center}
\includegraphics[width = 3in]{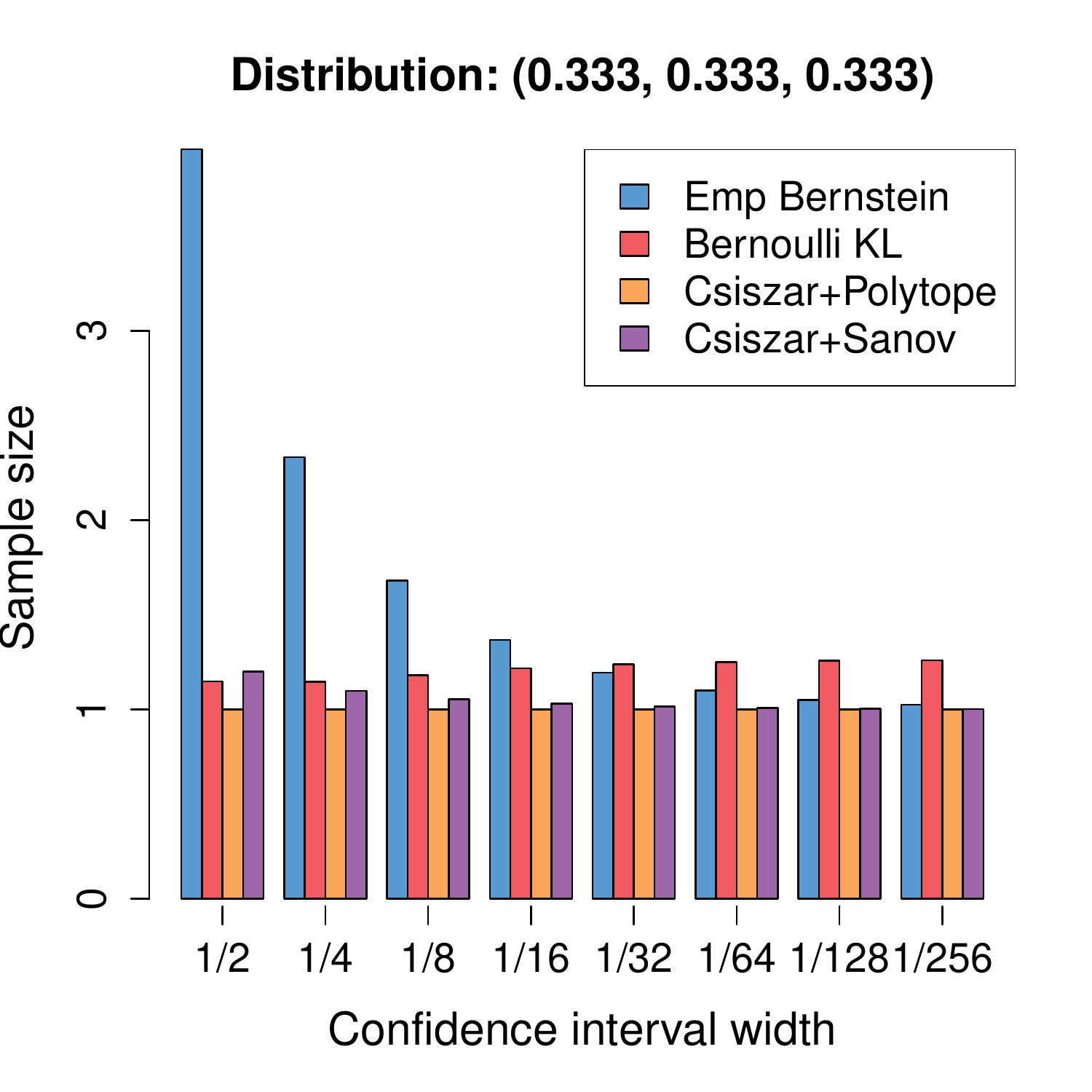}
\includegraphics[width = 3in]{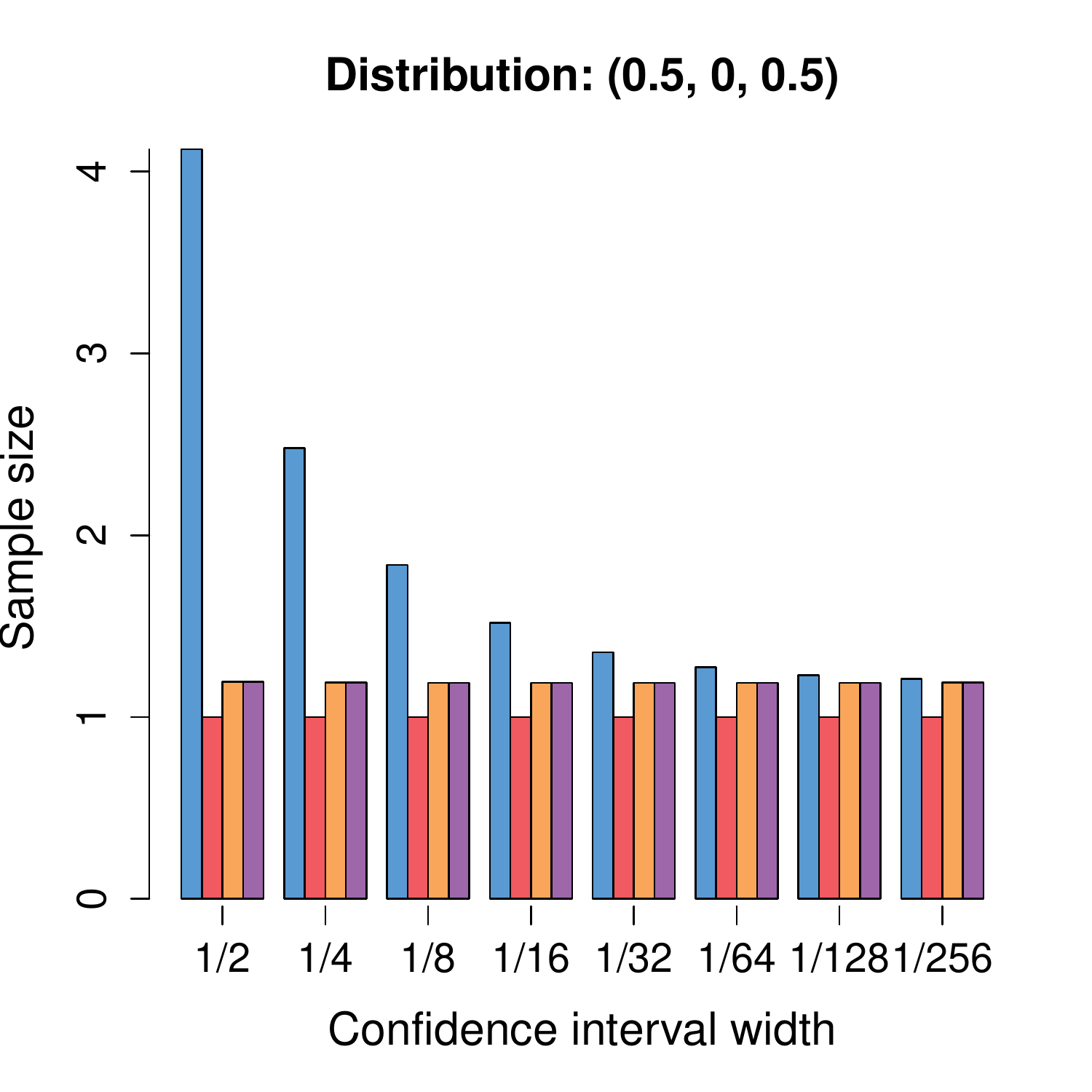}
\includegraphics[width = 3in]{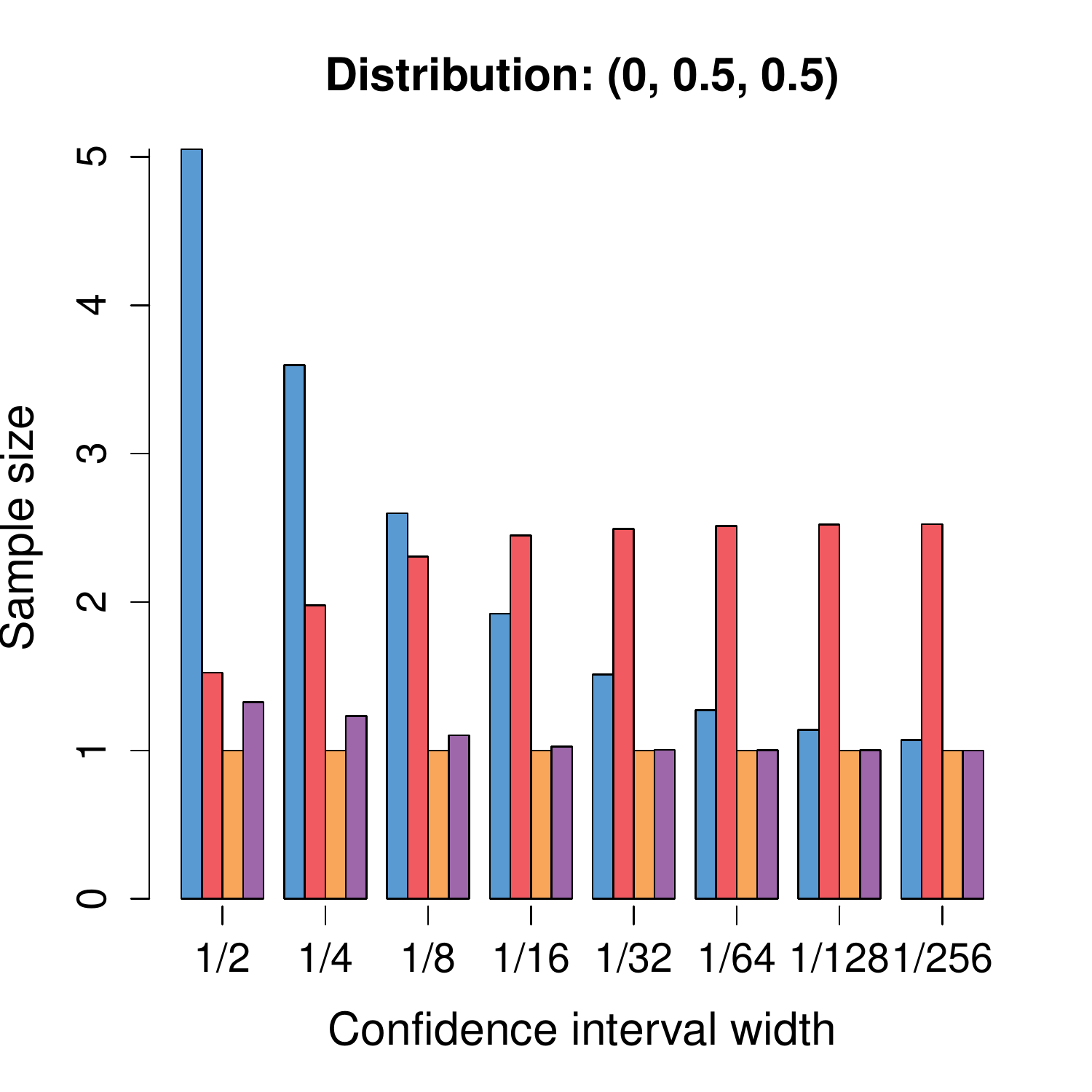}

\caption{Average sample size needed for the width of the confidence bound for the mean to reach a desired level, for various distributions. The high-level findings are similar for all cases: Empirical Bernstein (blue) performs poorly in the small sample regime (large interval width), but improves as the sample size increases. Bernoulli-KL (red) performs relatively well for small samples, but its performance deteriorates, unless the true distribution is Bernoulli, in which case it performs best. Our new bounds (orange and purple) perform best uniformly across all sample sizes, and have comparable performance to the Bernoulli-KL when the distribution is Bernoulli.\label{fig:linear_k3}}
\end{center}
\end{figure*}

\begin{figure*}[h]
\begin{center}
\includegraphics[width = 3in]{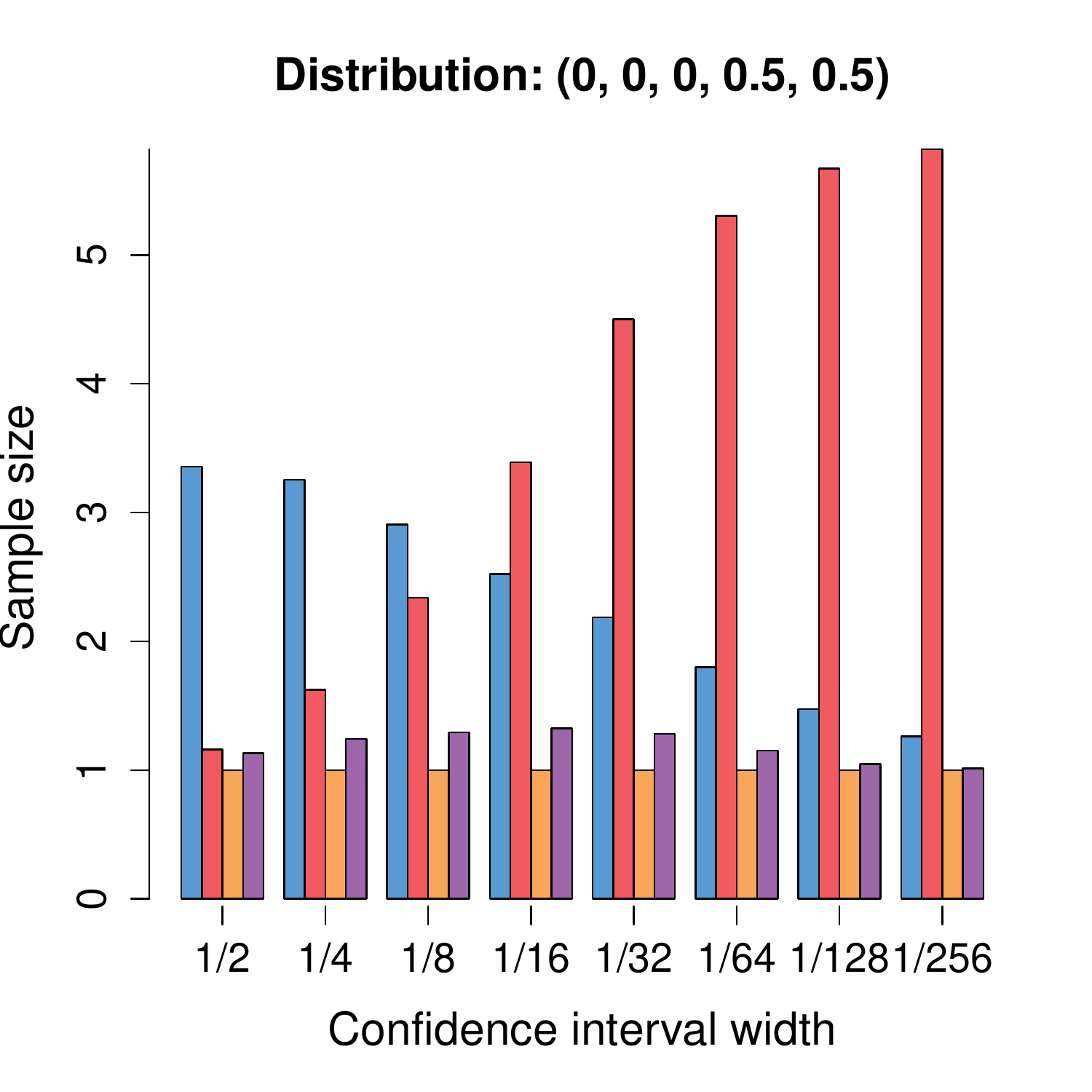}
\includegraphics[width = 3in]{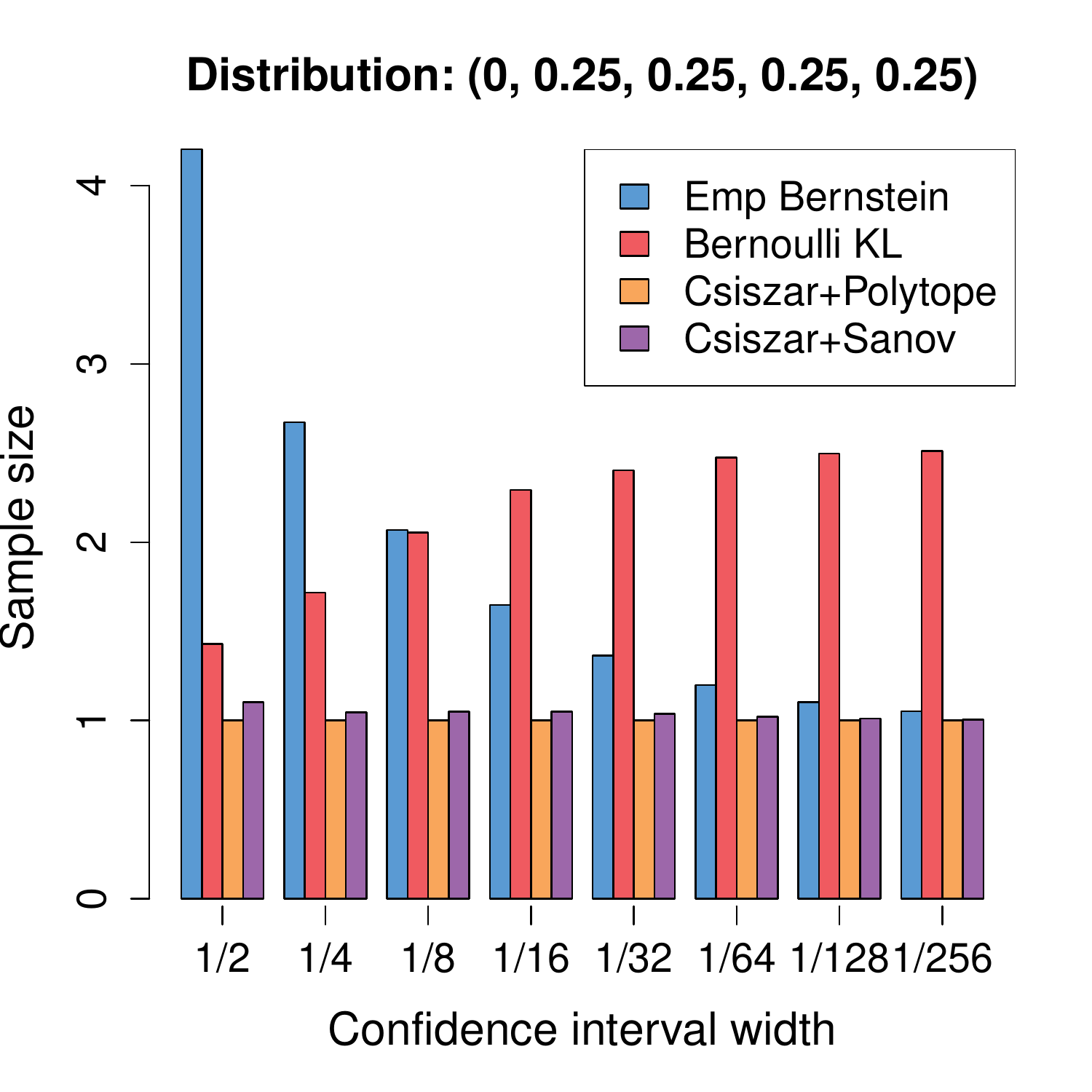}

\caption{Average sample size needed for the width of the confidence bound for the mean to reach a desired level, for various distributions. The high-level findings are similar for all cases: Empirical Bernstein (blue) performs poorly in the small sample regime (large interval width), but improves as the sample size increases. Bernoulli-KL (red) performs relatively well for small samples, but its performance deteriorates, unless the true distribution is Bernoulli, in which case it performs best. Our new bounds (orange and purple) perform best uniformly across all sample sizes, and have comparable performance to the Bernoulli-KL when the distribution is Bernoulli..\label{fig:linear_k5}}
\end{center}
\end{figure*}

\subsection{Quantiles}

The larger the alphabet size $k$, potentially the bigger problem the union bound becomes when using the KL-Bernoulli CDF bounds. We present similar numerical experiments to those in the main body of the paper, but for $k=10$. The results tell a similar story: the performance of the KL-based bounds is not much worse than the DKWM near the median, but get much better for quantiles far from the median.

\begin{figure*}[h]
\begin{center}
\includegraphics[width = 3in]{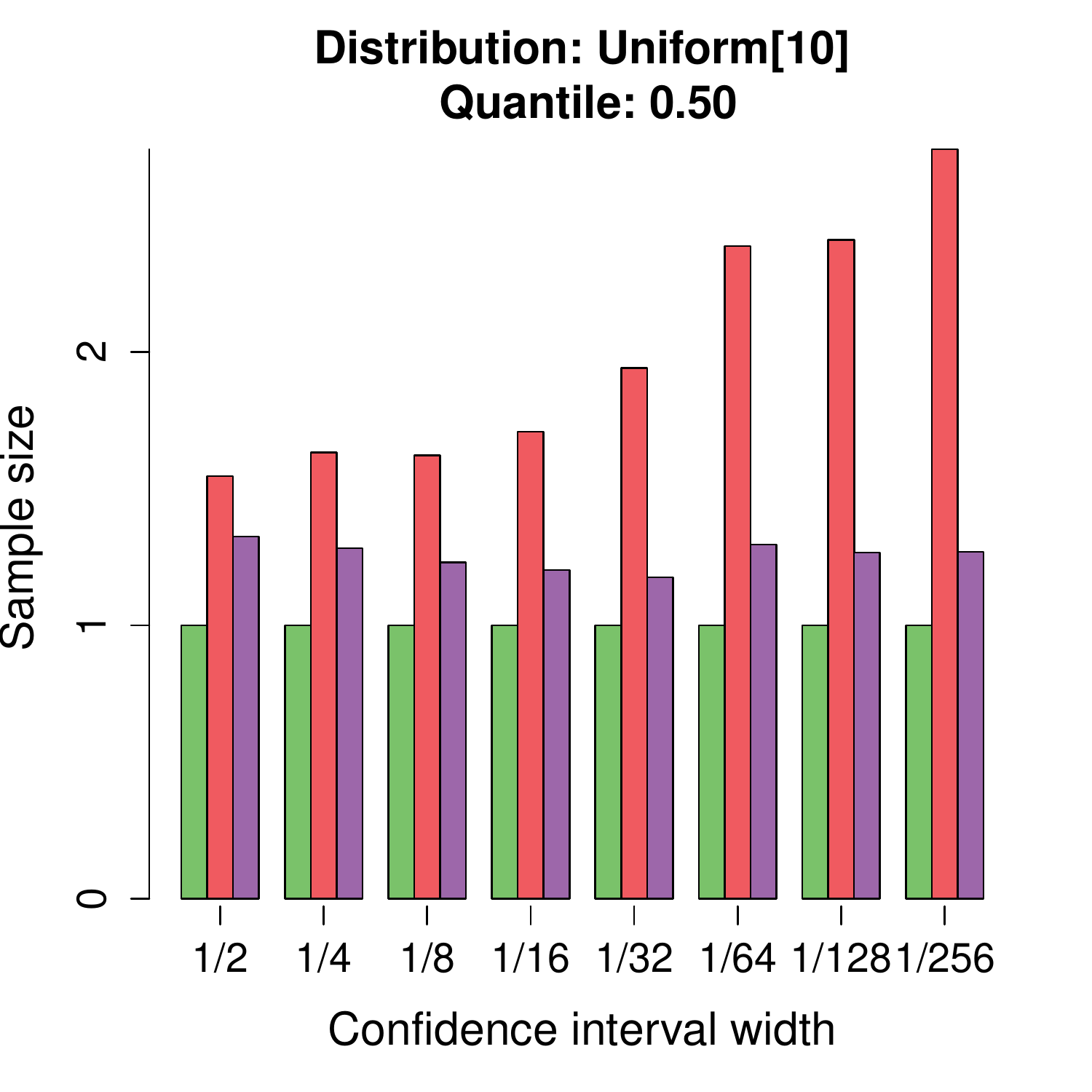}
\includegraphics[width = 3in]{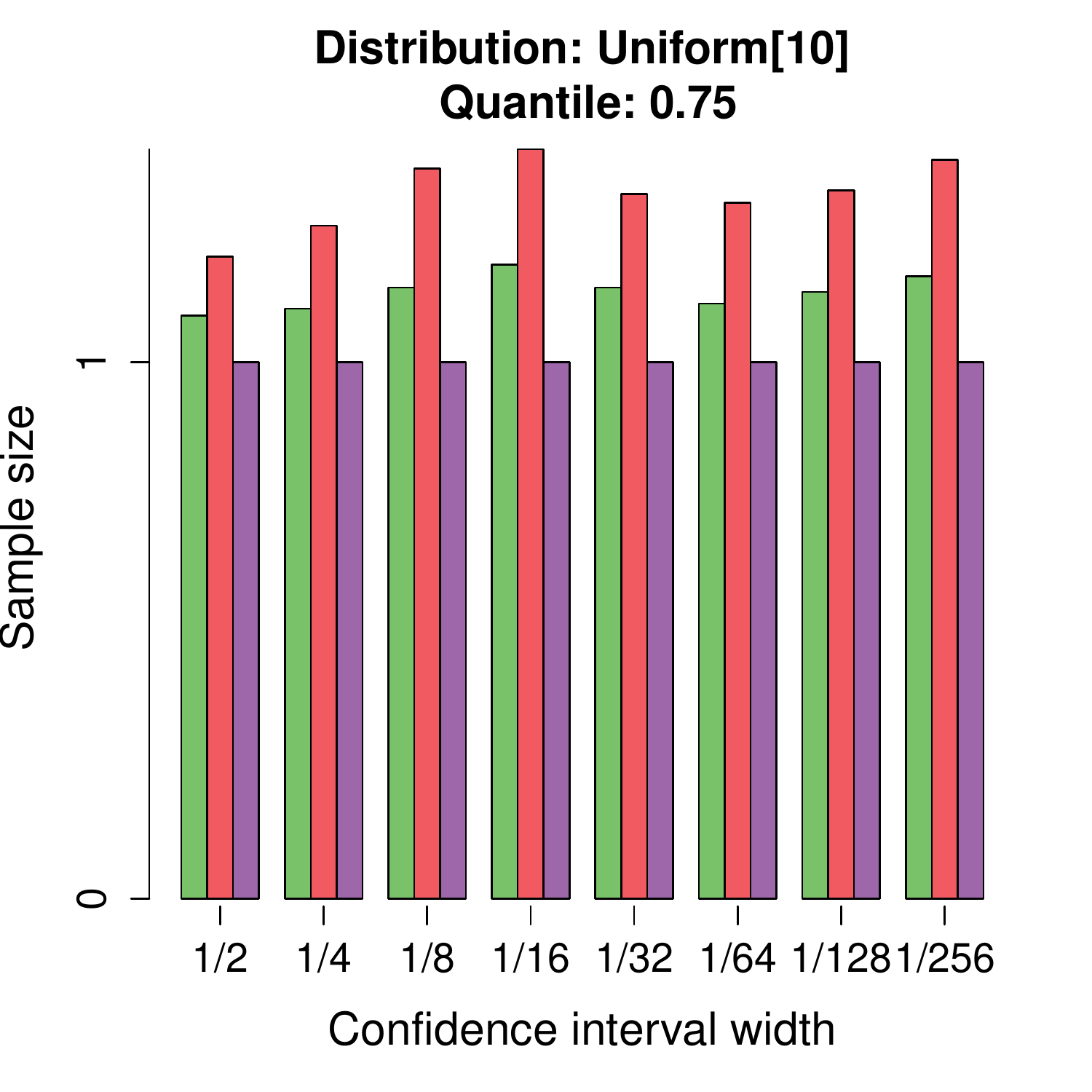}
\includegraphics[width = 3in]{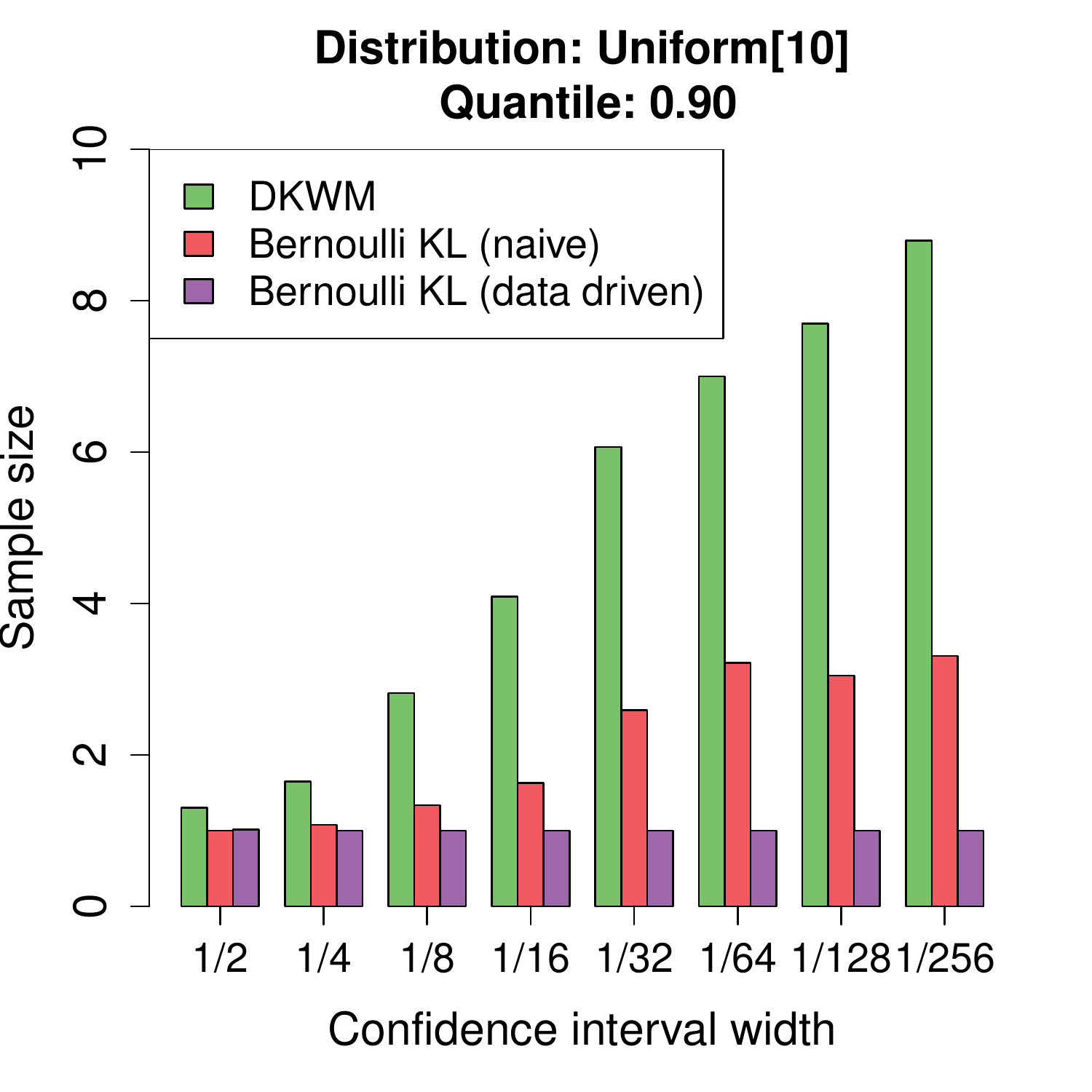}

\caption{Average sample size needed for the width of the confidence bound for the quantile to reach a desired level, for various quantiles. The true distribution is $\Unif [10]$ in all cases. The Bernoulli-KL bound with a data-driven union bound (purple) shows better performance compared to the one with a naive union bound (red)across the board. The figures indicate comparable performance between the DKWM bound (green) and the Bernoulli-KL bound with a data-driven union bound (purple) for quantiles around the median. However, for the 90\% quantile, the Bernoulli KL bounds clearly outperform the DKWM bounds.\label{fig:quantiles_k10}}
\end{center}
\end{figure*}

\end{document}